\documentclass[11pt]{article}

\usepackage{amsfonts,amsthm,amsmath,amssymb}
\usepackage{graphicx}
\usepackage{mathtools,booktabs,bm,bbm,stmaryrd}
\usepackage{color}

\allowdisplaybreaks

\newcommand{\sq}{\varepsilon}
\newcommand{\ssq}{\sqrt{\varepsilon}}

\newcommand{\uN}{U}
\newcommand{\pN}{P}
\newcommand{\qN}{Q}
\newcommand{\wN}{\bm{W}}

\newcommand{\vph}{\bm{\mathbbm{z}}}
\newcommand{\vphu}{\mathbbm{v}}
\newcommand{\vphq}{\mathbbm{r}}
\newcommand{\vphp}{\mathbbm{s}}
\newcommand{\dual}[2]{\left\langle#1,#2\right\rangle}
\newcommand{\norm}   [1] {\left\Vert#1\right\Vert}

\newcommand{\enorm}   [1] {\interleave #1\interleave_{E}}
\newcommand{\lnorm}   [1] {\interleave #1\interleave_{2}}

\newcommand{\jump}   [1] {[\![#1]\!]}

\newcommand{\prou}{\Pi^-}
\newcommand{\prop}{\Pi_x^+}
\newcommand{\proq}{\Pi_y^+}
\newcommand{\sobv}{\mathcal{V}_N}
\newcommand{\Bln}[2]{B(#1;#2)}
\newcommand{\spc}{\mathcal{V}_N}

\newtheorem  {theorem}    {\hspace{15pt} Theorem}[section]
\newtheorem  {remark}     {\hspace{15pt} Remark}[section]

\newtheorem  {lemma}     {\hspace{15pt} Lemma}[section]
\newtheorem  {assumption}     {\hspace{15pt} Assumption}[section]

\title{The local discontinuous Galerkin method for\\
 a singularly perturbed convection-diffusion problem\\
    with characteristic and exponential layers\footnotemark[1]}

\author{
Yao Cheng\footnotemark[2]
\quad
Martin Stynes\footnotemark[3]
}

\begin{document}

\maketitle

\renewcommand{\thefootnote}{\fnsymbol{footnote}}

\footnotetext[1]
{This paper was partly written during a visit 
by Yao Cheng to Beijing Computational Science Research Center  
in July 2022. The research of Y. Cheng was supported by NSFC grant 11801396 
and Natural Science Foundation of Jiangsu Province grant BK20170374.
The research of M. Stynes was partly supported by
NSFC grants 12171025 and NSAF-U1930402.
}

\footnotetext[2]
{School of Mathematical Sciences,
	Suzhou University of Science and Technology,
	Suzhou 215009, Jiangsu Province, China;
	ORCID 0000-0001-9500-9398,
	\textbf{ycheng@usts.edu.cn} 
}

\footnotetext[3]
{Corresponding author. Applied and Computational Mathematics Division, Beijing Computational Science Research Center, 
	Beijing 100193, China; 
	ORCID 0000-0003-2085-7354, 
	\textbf{m.stynes@csrc.ac.cn}
}


\begin{abstract}
A singularly perturbed convection-diffusion problem,
posed on the unit square in $\mathbb{R}^2$,  is studied;
its solution has
both exponential and characteristic boundary layers.
The problem is solved numerically using the local discontinuous Galerkin (LDG) method on 
Shishkin meshes.
Using tensor-product piecewise  polynomials of degree at most $k>0$
in each variable, 
the error between the LDG solution and the true solution
is proved to converge,
uniformly in the singular perturbation parameter,
at a rate of $O((N^{-1}\ln N)^{k+1/2})$
in an associated energy norm, where
$N$ is the number of mesh intervals in each coordinate direction.
(This is the first uniform convergence result proved for the LDG method applied to a problem with characteristic boundary layers.)
Furthermore, we prove that this order of convergence 
increases to $O((N^{-1}\ln N)^{k+1})$
when one measures the energy-norm difference between the LDG solution 
and a local Gauss-Radau projection of the true solution into the finite element space.
This uniform supercloseness property implies an optimal $L^2$ error estimate of 
order $(N^{-1}\ln N)^{k+1}$ for our LDG method. 
Numerical experiments show the sharpness of  our theoretical results.

\end{abstract}



\section{Introduction}
\label{intro}

Consider the singularly perturbed convection-diffusion problem
\begin{subequations}\label{cd:spp:parabolic}
	\begin{align}
	-\varepsilon \Delta u + a(x,y) u_x + b(x,y) u
	&= f(x,y)  \quad \text{ in }  \Omega=(0,1)\times (0,1),   \label{spp:pde}
	\\
	u &= 0     \quad \text{ on }  \partial\Omega,  \label{spp:bc}
	\end{align}
\end{subequations}
where $\varepsilon>0$ is a small parameter, 
$a$, $b$ and $f$ are sufficiently smooth,
$a\geq \alpha$, $b\geq 0$ on $\overline{\Omega}$.
Here $\alpha$ is a positive constant.
Assume that 
\begin{equation}\label{assumption:coef:2d}
b- \frac12 a_x \geq \beta>0, \qquad (x,y)\in \overline{\Omega}
\end{equation}
for some positive constant $\beta$.
With these assumptions, it is straightforward to use the Lax-Milgram lemma to show that the problem \eqref{cd:spp:parabolic}
has a unique solution in $H_0^1(\Omega)\cap H_2(\Omega)$.
Note that for the small $\varepsilon>0$,
condition (\ref{assumption:coef:2d}) can always be ensured
by a simple transformation
$u(x,y)=e^{tx}v(x,y)$ with a suitably chosen positive constant~$t$. 

The problem \eqref{cd:spp:parabolic} is of practical importance; for example, it can be used in modelling the flow past a surface with a no-slip boundary condition. Thus the numerical analysis of this problem can help our understanding of the behaviour of numerical methods in the presence of layers in more complex problems such as the linearised Navier-Stokes equations at high Reynolds number \cite{RST2008}. Compared to  2D problems whose solutions exhibit  only exponential layers (which are essentially one-dimensional in nature), the solution of problem \eqref{cd:spp:parabolic} has a much more complicated analytical structure: a regular exponential boundary layer, two characteristic (or parabolic) boundary layers and various types of corner layer functions in the vicinity of the inflow and outflow corners
 \cite{Kellogg:Stynes:2003:JDE,Kellogg:Stynes:2007:AML,StySty18}.

Standard numerical methods on quasi-uniform meshes do not produce satisfactory numerical approximations for singularly perturbed problems like \eqref{cd:spp:parabolic} unless the mesh diameter is comparable to the small parameter $\varepsilon$. Consequently layer-adapted meshes, such as Shishkin meshes, Bakhvalov meshes and their generalisations, have been developed; see \cite{Linss10} for an overview of these. On these meshes, uniform convergence is often observed \cite{Linss:Stynes:2001} and analysed \cite{Andreev2010,Brdar2021,Franz:Kellogg:Stynes2012,LiuZhang2018,Riordan2008}. 
Here and subsequently, \emph{``uniform" means that the bound on the error in the numerical solution is independent of the value of the singular perturbation parameter~$\varepsilon$.}

\subsection{Discontinuous Galerkin methods}\label{sec:dGmethods}

The standard Galerkin method is applied to problem \eqref{cd:spp:parabolic} in \cite{LiuZhang2018},
while stabilised finite element methods for the problem 
--- which reduce oscillations in computed solutions --- 
are considered in several papers, e.g., the 
streamline diffusion finite element method \cite{Brdar2021}, 
continuous interior penalty method \cite{Zhang:Stynes:2017}, 
local projection stabilization \cite{Franz2010}
and interior penalty discontinuous Galerkin method \cite{Zarin2005}. 
In these papers uniform convergence and uniform supercloseness results are derived, 
but with the exception of~\cite{Franz2010}, these results
are confined to low-order elements.

The local discontinuous Galerkin (LDG) method is a stabilised finite element method that was originally proposed for convection-diffusion systems \cite{Cockburn:Shu:LDG}. It is stabilised by adding only jump terms at element boundaries into the bilinear form. Its weak stability 
and local solvability are advantageous when solving problems with singularities (such as layers). Even on a uniform mesh, the LDG method applied to a singularly perturbed problem does not produce an oscillatory solution;  see the numerical experiments in~\cite{CYWL22,Xie2009JCM}. 

For 1D convection-diffusion problems, various uniform convergence results for the LDG solution were derived in \cite{Xie2010MC,Zhu:2018}. 
For 2D convection-diffusion problems that have only exponential boundary layers, when the LDG method is used to compute solutions on layer-adapted meshes, uniform convergence in an energy norm  has been established in \cite{Cheng2021:Calcolo,Cheng2020,Zhu:2dMC} and 
uniform supercloseness of the LDG solution is proved in the recent paper~\cite{Cheng:Jiang:Stynes:2022}.

\subsection{Theoretical difficulties with characteristic layers}\label{sec:charlayers}

Exponential layers are essentially one-dimensional in nature, but characteristic layers are intrinsically two-dimensional (see, e.g., \cite[Section~4.1]{StySty18}) and so are more difficult to analyse and to approximate numerically. Thus, in the numerical analysis literature one can find many papers that analyse the error when solving a convection-diffusion problem whose solution has exponential layers, but far fewer papers that perform an error analysis for the numerical solution of a problem with characteristic layers. Indeed, for discontinuous Galerkin methods in general, the only papers that derive a uniform error bound for characteristic layers seem to be those of Roos and Zarin, who use a nonsymmetric discontinuous Galerkin finite element method with interior penalties; see~\cite{Zarin2005} and its references. In a subsequent paper \cite{Roos2007}  by Roos and Zarin that also uses this method,  uniform supercloseness is proved for a problem with exponential layers, but the authors comment that ``We conjecture that the analysis can be extended to problems with characteristic layers but this requires a careful study of anisotropic elements in the region where these layers occur." The message here is that error analyses that work for exponential layers are not easily modified to work for characteristic layers.

Returning to the LDG method, despite the positive results for exponential layers that were described in Section~\ref{sec:dGmethods}, we are not aware of any uniform convergence result for this method applied on a layer-adapted mesh to solve a problem whose solution contains characteristic layers, even though it is 8 years since the exponential-layer uniform  error analysis of~\cite{Zhu:2dMC} appeared. \emph{Our paper will prove  uniform convergence of the LDG method for both exponential and characteristic layers --- in fact it goes further by deriving a higher-order  uniform supercloseness result for the error between the true solution and a projection of it into the finite element space.}
(It should be noted that uniform supercloseness cannot be obtained by following the line of analysis in \cite{Cheng2021:Calcolo,Cheng2020,Zhu:2dMC}, since these papers use a streamline-diffusion type norm that yields suboptimal estimates; our analysis below is based on a discrete energy norm that fits naturally with the LDG method.)

\subsection{Technical innovations in the analysis}\label{sec:tech}

As one can infer from the discussion of Section~\ref{sec:charlayers}, several technical innovations are needed to obtain a uniform error analysis for a problem such as \eqref{cd:spp:parabolic} whose solution contains characteristic boundary layers. 
\begin{itemize}
\item A fundamental difference between our error analysis and that of~\cite{Zhu:2dMC} (who consider only exponential layers) is that we use a unified Gauss-Radau projection of the solution~$u$ instead of a combination of standard Lagrange interpolation and Gauss-Radau projection. 
This choice enables us to treat the LDG error in the smooth component of the true solution in an optimal way,
and to avoid using inverse inequalities such as \cite[eqs.~(4.35) and (4.36)]{Zhu:2dMC} which lead to suboptimal results.

\item Our analysis does reuse some ideas from~\cite{Cheng:Jiang:Stynes:2022}, which considers only exponential layers, but an immediate significant difference is that our analysis here requires the crosswind stabilisation parameter $\lambda_2$ in the LDG method to be positive, while $\lambda_2=0$ was permitted in~\cite{Cheng:Jiang:Stynes:2022}.  
We also treat in a delicate way the error component $\mathcal{T}_4$ in~\eqref{error:equation} that is associated with the convection term $au_x$ of~\eqref{spp:pde},  by applying different techniques in different subregions: when the mesh is coarse in the convective direction, we mainly employ a superapproximation propery of the local bilinear form, but where the mesh is fine in the convective direction, we make full use of a crucial stability property ---  that the derivative and jump of the error in the finite element space are controlled by the energy-norm itself plus a term with optimal convergence rate. All of these devices, working together, yield our uniform supercloseness estimate for the approximation error in the finite element space.

\item The new challenge of characteristic layers means that special attention has to be paid to the treatment of various solution components on different parts of the domain, 
and new bounds have to be established for errors in approximating derivatives in the crosswind direction; see \eqref{etaq:bry:top}--\eqref{etaq:L2} of Lemma~\ref{lemma:GR2} and \eqref{sup:y} of Lemma~\ref{superapproximation:element}.  

\item Finally, it should be noted that our analysis includes finite elements with piecewise polynomials of any positive degree, whereas~\cite{Zarin2005} (the only previous uniform convergence result for a discontinuous Galerkin method applied to a problem with characteristic layers) is for bilinears only.
\end{itemize}

\subsection{Detailed results}\label{sec:detailedresults}

Our paper will analyse in detail the convergence behaviour of the 
LDG method on Shishkin meshes applied to the problem~\eqref{cd:spp:parabolic}. Using piecewise polynomials of degree at most $k$ (an arbitrary positive integer) in each coordinate variable, we shall prove that on a suitable Shishkin mesh, the LDG solution converges uniformly
to the true solution in the energy norm induced by the LDG bilinear form at the rate 
$O((N^{-1}\ln N)^{k+1/2})$, where
$N$ is the number of mesh intervals in each coordinate direction.
We also establish an enhanced energy-norm uniform convergence rate of $O((N^{-1}\ln N)^{k+1})$ for the difference between the numerical solution and the local Gauss-Radau projection of the true solution into the finite element space. 
This uniform supercloseness property implies that the $L^2$ error between the numerical and true solutions achieves 
the optimal uniform convergence rate  $O((N^{-1}\ln N)^{k+1})$. Numerical experiments show the sharpness of all these error bounds. 
 
 To obtain these high orders of approximation (recall that $k$ is any positive integer), it is necessary to assume that the true solution 
 of~\eqref{cd:spp:parabolic}  possesses a sufficient degree of regularity; see Section~\ref{sec:decomp}.  
This assumption is reasonable, given enough smoothness and compatibility of the problem data, 
but to derive it rigorously would demand a significant amount of extra analysis. 

To make the paper more readable and concise we have considered only Shishkin meshes in detail, but the entire analysis can be extended to  other families of layer-adapted meshes as we describe in Remark~\ref{Remark:B:BS:mesh}; see also the numerical results of Section~\ref{subsec:BS:B:mesh}. 

Our paper is organised as follows. In Section~\ref{sec:scheme}, 
we discuss a decomposition of the solution of~\eqref{cd:spp:parabolic},
construct the Shishkin mesh and define the LDG method.
Section~\ref{sec:GRproj} is devoted to the definition and properties of 
the local Gauss-Radau projection that we use, followed by a lengthy derivation of bounds on
various measures of the error between the true solution of~\eqref{cd:spp:parabolic} 
and its Gauss-Radau projection.
Then Section~\ref{sec:analysis} is the heart of the paper, where we carry out 
uniform convergence and uniform supercloseness analyses of the error in the LDG solution.   
In Section~\ref{sec:experiments},  
we present numerical experiments to demonstrate the sharpness of our theoretical error bounds. 
Finally,  Section~\ref{sec:conclusion} gives some concluding remarks.

\emph{Notation.} We use $C$ to denote a generic positive constant that may depend 
on the data $a, b,f$ of~\eqref{cd:spp:parabolic}, the parameters $\sigma,\delta$ of~\eqref{tau}, 
and the degree $k$ of the polynomials in our finite element space,  
but is independent of $\varepsilon$ and of $N$ (the number of mesh intervals in each coordinate direction); 
$C$ can take different values in different places.

The usual Sobolev spaces $W^{m,\ell}(D)$ and $L^{\ell}(D)$ will be used, 
where $D$ is any one-dimensional interval subset of $[0,1]$ or
any measurable two-dimensional subset of~$\Omega$. 
The $L^{2}(D)$ norm is denoted by $\norm{\cdot}_{D}$,
the $L^{\infty}(D)$ norm by $\norm{\cdot}_{L^\infty(D)}$,
and $\dual{\cdot}{\cdot}_D$ denotes the $L^2(D)$ inner product.
The subscript $D$ will always be dropped when $D= \Omega$.
We set $\partial_x^i\partial_y^j:=\frac{\partial^{i+j}}{\partial x^i \partial y^j}$ for all nonnegative integers $i$ and $j$.

\section{Solution decomposition, Shishkin mesh and LDG method}
\label{sec:scheme}
In this section we assemble the basic tools for the construction and analysis of our numerical method. 

\subsection{Solution decomposition}\label{sec:decomp}
Typical solutions $u$ of \eqref{cd:spp:parabolic} have 
an exponential layer along the side $x=1$
and characteristic layers along the two sides $y=0$ and $y=1$;
see, e.g., \cite[Example 4.2]{StySty18}. A solution may also have a corner layer
at each corner of~$\Omega$. In \cite{Kellogg:Stynes:2003:JDE,Kellogg:Stynes:2007:AML}
the case of constant $a$ and $b$ is fully analysed and a decomposition of 
the solution $u$ into a smooth component plus layers of various types is constructed, 
together with bounds on their derivatives of all orders.
In \cite{Riordan2008} the case of variable $a(x,y)$ with $b\equiv 0$ is discussed
under corner compatibility assumptions that essentially exclude the corner layers at the inflow corners 
$(0,0)$ and $(0,1)$, and $u$ is again decomposed into a sum of a smooth component, 
an exponential boundary layer along $x=1$, two parabolic layers along $y=0$ and $y=1$, and
two corner layers at the outflow corners $(1,0)$ and $(1,1)$, 
for each of which bounds on certain low-order derivatives are obtained.

It is an open question whether one can decompose the solution $u$ of \eqref{cd:spp:parabolic} 
in a similar way together with bounds on high-order derivatives of each component in the decomposition
(this would extend the work described in the previous paragraph). 
It seems reasonable that this is indeed the case, so like \cite[Assumption 2.1]{Brdar2021} 
and \cite[Part III: Section 1.4]{RST2008}
we shall now assume that  this decomposition is possible. 

\begin{assumption}\label{assumption:reg:2d}
	Let $m$ be a non-negative integer. Let $\kappa$ satisfy $0<\kappa<1$.
	Under suitable smoothness and compatibility conditions on the data,
	the solution $u$ of \eqref{cd:spp:parabolic} 
	lies in the H\"older space $C^{m+2,\kappa}(\Omega)$ 
	and can be decomposed as 	$u = u_0 +u_{1}+u_{2}+u_{12}$, where
	$u_0$ is the smooth component, $u_1$ is the exponential layer, $u_2$ is
	the sum of the two parabolic layers, and $u_{12}$ is the sum of the two outflow corner layers. 
	The derivatives of each of these components satisfy the following bounds 
	for all $(x,y)\in \bar{\Omega}$ and all nonnegative integers $i,j$ with $i+j\le m+2$: 
	\begin{subequations}\label{reg:u:2d}
		\begin{align}
		\label{reg:u0}
		\left| \partial_x^{i}\partial_y^{j}u_0(x,y)\right|
		\leq&\; C,
		\\
		\label{reg:u1}
		\left| \partial_x^{i}\partial_y^{j}u_{1}(x,y)\right|
		\leq &\; C\varepsilon^{-i} e^{-\alpha(1-x)/\varepsilon},
		\\
		\label{reg:u2}
		\left| \partial_x^{i}\partial_y^{j}u_{2}(x,y)\right|
		\leq &\; C\varepsilon^{-j/2} \left(e^{-\delta y/\ssq}
		+e^{-\delta(1-y)/\ssq}\right),
		\\
		\label{reg:u12}
		\left| \partial_x^{i}\partial_y^{j}u_{12}(x,y)\right|
		\leq &\; C\varepsilon^{-(i+j/2)}e^{-\alpha(1-x)/\varepsilon}
		\left(e^{-\delta y/\ssq}+e^{-\delta(1-y)/\ssq}\right),
		\end{align}
	\end{subequations}
	where $C>0$ and $\delta>0$ are some constants.
\end{assumption}

\begin{remark} 
In the numerical analysis that follows, we shall need $m=k$ in Assumption~\ref{assumption:reg:2d}, 
where $k$ is the degree of the piecewise polynomials in our finite element space.  
\end{remark} 

\subsection{The Shishkin mesh}
\label{subsec:layer:adapted:meshes}

We shall use a piecewise uniform Shishkin mesh
\cite{Linss10,StySty18} that is refined near the sides $x=1$, $y=0$ and $y=1$ of~$\Omega$. 
Define the \emph{mesh transition parameters} as
\begin{equation}  \label{tau}
\tau_1 := \min\left\{\frac12, 
\frac{\sigma\varepsilon}{\alpha} \ln N\right\}
\quad \text{and} \quad
\tau_2 := \min\left\{\frac14, 
\frac{\sigma\ssq}{\delta} \ln N
\right\},
\end{equation}
where $\sigma>0$ is a user-chosen parameter 
whose value affects our error estimates; 
in the error analysis
it will be seen that~$\sigma$ needs to be sufficiently large. 
Assume in~\eqref{tau}  that 
\begin{equation}\label{tau:used}
\tau_1 = \frac{\sigma\varepsilon}{\alpha} \ln N
\leq \frac12
\text{   and   }
\tau_2 = \frac{\sigma\ssq}{\delta} \ln N 
\leq \frac14,
\end{equation}
as is typically the case for \eqref{cd:spp:parabolic}.
Note that \eqref{tau:used} implies a mild assumption $\ssq\ln N\leq C$ 
which will be used frequently in the following analysis.
(We remark that the stronger inequality $\varepsilon \leq N^{-1}$ is used in \cite[eq.~(4.35)]{Zhu:2dMC}.)

Let $N\geq 4$ be an integer divisible by 4. 
Our meshes will use $N+1$ points in each coordinate direction.
Define the mesh points $(x_i,y_j)$ for $i,j=0,1,\dots, N$ by
\begin{equation}\label{mesh:points}
\begin{split}
x_i&:=
\begin{dcases}
2(1-\tau_1)\frac{i}{N}   &\qquad\text{for } i=0,1,...,N/2,
\\
1-2\tau_1\left(1-\frac{i}{N}\right)
&\qquad\text{for } i=N/2+1,N/2+2,...,N,
\end{dcases}
\\
y_j&:=
\begin{dcases}
4\tau_2\frac{j}{N}   
&\text{for } j=0,1,...,N/4,
\\
\tau_2 + 2(1-2\tau_2)\left(\frac{j}{N}-\frac14\right)
&\text{for } j=N/4+1,N/4+2,...,3N/4,
\\
1-4\tau_2\left(1-\frac{j}{N}\right)
&\text{for } j=3N/4+1,3N/4+2,...,N.
\end{dcases}
\end{split}
\end{equation}

The Shishkin mesh $\Omega_N$ is then constructed by drawing axiparallel lines through the mesh points $(x_i,y_j)$,
i.e., set $\Omega_N:=\{K_{ij}\}_{i, j=1,\dots,N}$,
where each rectangular mesh element 
$K_{ij} :=I_i\times J_j :=(x_{i-1},x_i)\times (y_{j-1},y_j)$.

Figure \ref{fig:division} displays the mesh for 
$\varepsilon = 10^{-2}$, $N=8$, $\sigma/\alpha=4$ and $\sigma/\delta=1$
in \eqref{tau:used} and~\eqref{mesh:points};
it is uniform and coarse on $\Omega_{11}:=[0,1-\tau_1]\times[\tau_2,1-\tau_2]$, 
but is refined in the characteristic layer region
$\Omega_{12}:=[0,1-\tau_1]\times([0,\tau_2]\cup [1-\tau_2,1])$,
the  exponential layer region
$\Omega_{21}:=[1-\tau_1,1]\times[\tau_2,1-\tau_2]$,
and the corner layer region
$\Omega_{22}:=[1-\tau_1,1]\times([0,\tau_2]\cup [1-\tau_2,1])$.

Set $h_{x,i} = x_i-x_{i-1}$
and $h_{y,j} =  y_j-y_{j-1}$ for $i,j=1,2,\dots, N$. 
Then
\begin{equation}\label{mesh:size:x:y}
\begin{split}
h_{x,i}
&:=
\begin{cases}
\frac{2(1-\tau_1)}{N}=O(N^{-1}) 
&\text{for }i=1,2,\dots,N/2,\\
\frac{2\tau_1}{N}=O(\sq N^{-1}\ln N) 
&\text{for }  i=N/2+1,N/2+2,\dots,N,
\end{cases}
\\
h_{y,j}
&:=
\begin{cases}
\frac{2(1-2\tau_2)}{N}=O(N^{-1}) 
&\text{for } j=N/4+1,\dots,3N/4,\\
\frac{4\tau_2}{N}=O(\ssq N^{-1}\ln N)
&\text{for }  j=1,\dots,N/4\text{ and } j =3N/4+1,\dots,N.
\end{cases}
\end{split}
\end{equation}
These mesh sizes will be used frequently throughout our  analysis.

\begin{figure}[h]
	\begin{minipage}{0.495\linewidth}
		\centerline{\includegraphics[width=1.35\textwidth]{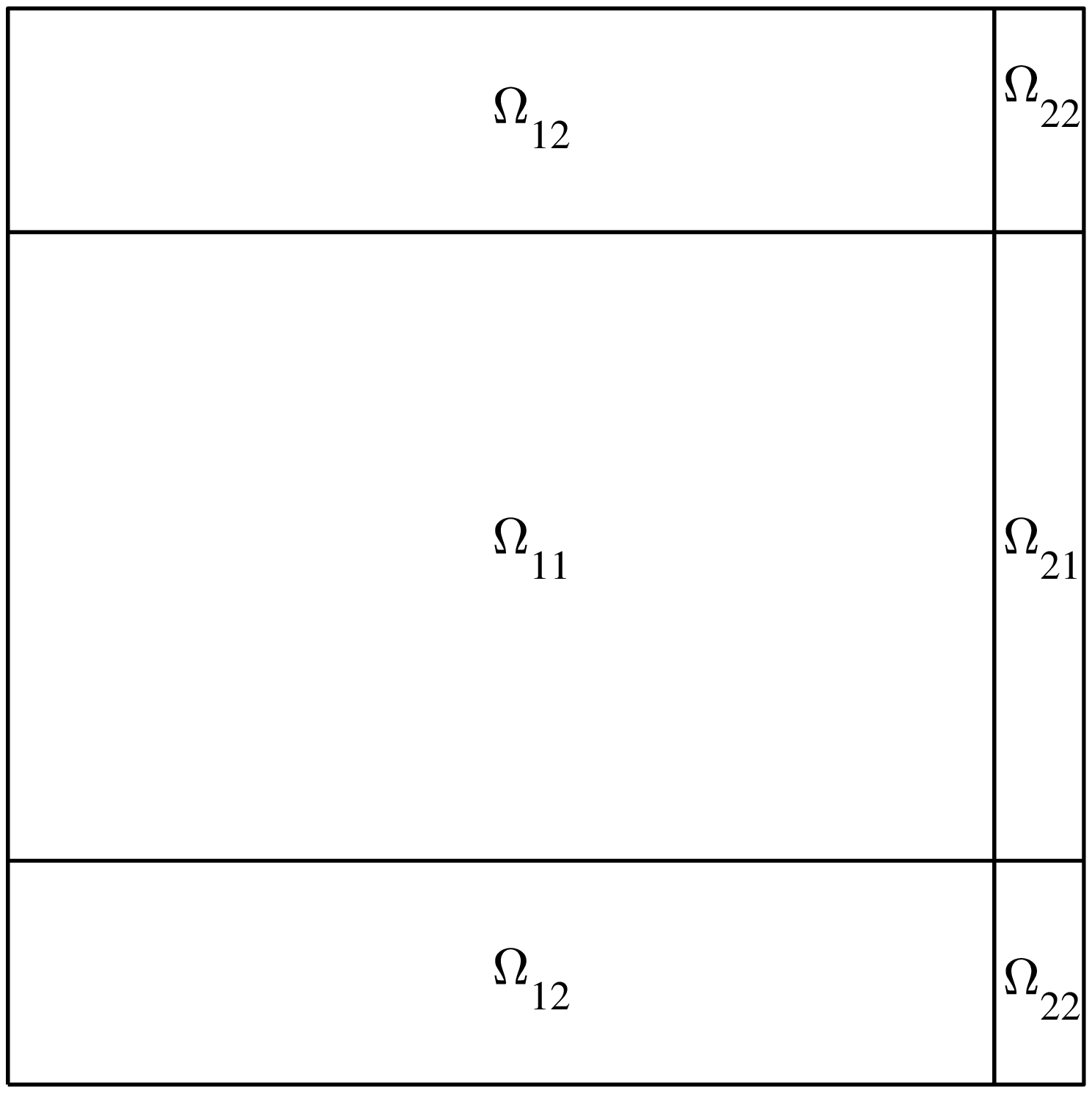}}
	\end{minipage}
    \begin{minipage}{0.495\linewidth}
		\centerline{\includegraphics[width=1.35\textwidth]{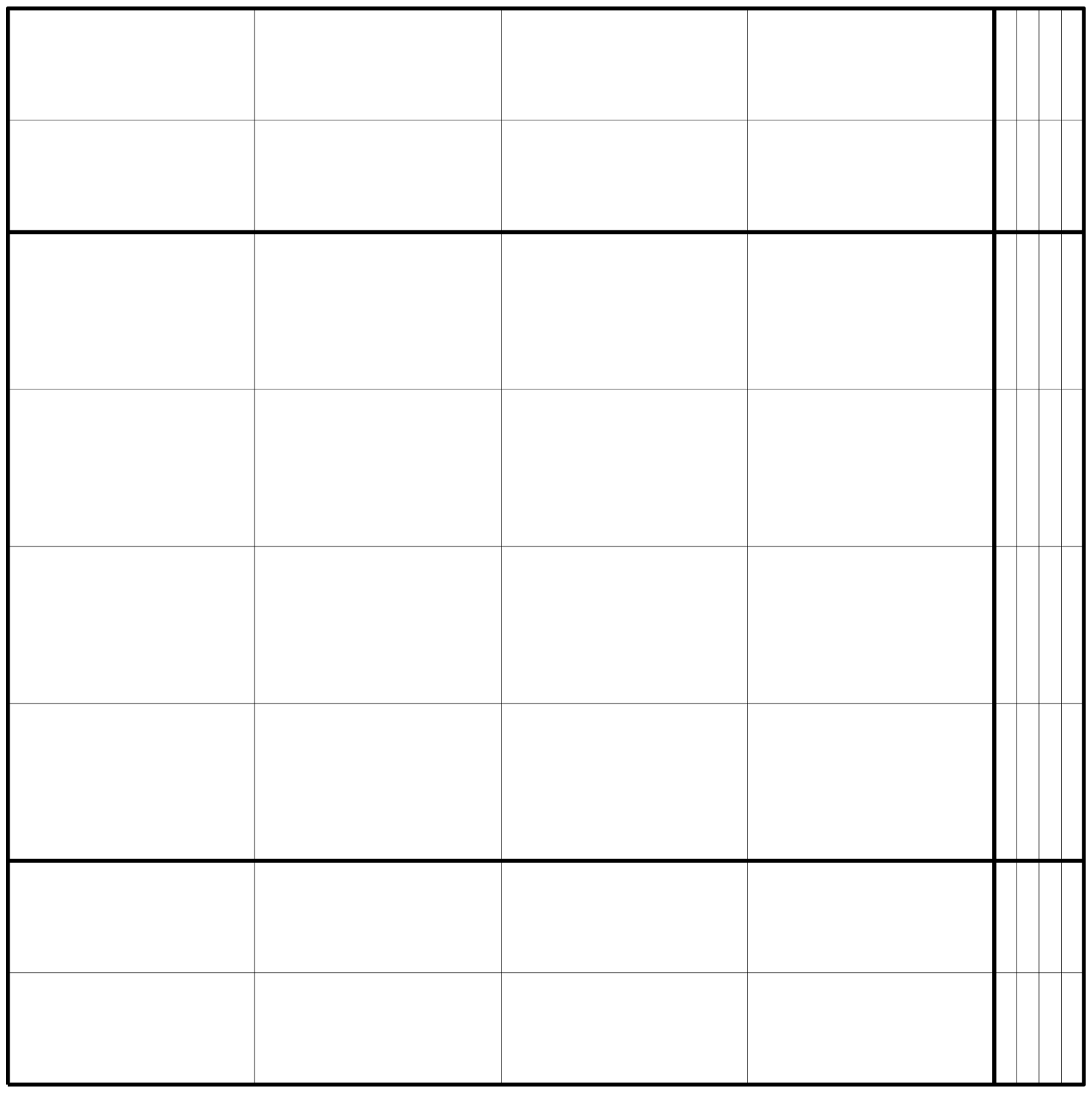}}
	\end{minipage}
	\caption{Domain division (left) and Shishkin mesh with $N=8$ (right)}.
	\label{fig:division}
\end{figure}

\subsection{The local discontinuous Galerkin (LDG) method}\label{sec:LDG}

Let $k$ be a fixed positive integer. On any 1-dimensional interval~$I$, 
let  $\mathcal{P}^{k}(I)$ denote the space of polynomials of degree at most~$k$ defined on~$I$.  
For each mesh element $K =  I_i \times J_j$, 
set $\mathcal{Q}^k(K) := \mathcal{P}^{k}(I_i)\otimes \mathcal{P}^{k}(J_j)$.
Then define the discontinuous finite element space  
\[
\spc= \left\{v\in L^2 (\Omega)\colon
v|_{K} \in \mathcal{Q}^{k} (K),  K\in\Omega_N\right\}.
\]
Note that functions in $\spc$ are allowed to 
be discontinuous across element interfaces.
For any $v\in \spc$ and $y\in J_j$, for $i=0,1,\dots,N$
we use $v^\pm_{i,y}=\lim_{x\to x_{i}^\pm}v(x,y)$
to denote the traces on vertical element edges; 
here in particular we set
$v^{-}_{0,y}=v^{+}_{N,y}=0$ to avoid treating $\partial\Omega$ as a special case.
The jumps on these edges are denoted by
$\jump{v}_{i,y}:= v^{+}_{i,y}-v^{-}_{i,y}$
for $i=0,1,\dots,N$; thus $\jump{v}_{0,y}:=v^{+}_{0,y}$
and $\jump{v}_{N,y}:=-v^{-}_{N,y}$.
In a similar fashion, define the traces $v^\pm_{x,j}$ and the jumps
$\jump{v}_{x,j}$ on the horizontal element edges for $j=0,1,\dots,N$.

To define the LDG method, rewrite \eqref{cd:spp:parabolic} as an  equivalent first-order system:
\[
-p_x-q_y+a u_x +bu=f,\quad p=\sq u_x,\quad q=\sq u_y
\]
with the homogeneous boundary condition of~\eqref{spp:bc}.
Then the final compact form of the LDG method reads as follows \cite{Cheng2020}:\\
\indent Find $\wN=(\uN,\pN,\qN)\in \spc^3 := \spc\times\spc\times\spc$ 
($U$ approximates~$u$, while $\pN$ and $\qN$ approximate $p$ and $q$  respectively)
such that
\begin{equation}\label{compact:form:2d}
	\Bln{\wN}{\vph}=\dual{f}{\vphu}
	\ \forall \vph=(\vphu,\vphp,\vphq)\in \spc^3,
\end{equation}
where 
\begin{align}
\label{B:def:2d}
B(\wN;\vph) &:=
\mathcal{T}_1(\wN;\vph)+\mathcal{T}_2(\uN;\vph)
+\mathcal{T}_3(\wN;\vphu)+\mathcal{T}_4(\uN;\vphu),
\end{align}
with
\begin{align}
\mathcal{T}_1(\wN;\vph)&=
\varepsilon^{-1}[\dual{\pN}{\vphp}+\dual{\qN}{\vphq}]
+\dual{(b- a_x)\uN}{\vphu},
\nonumber\\
\mathcal{T}_2(\uN;\vph)&=
\dual{ \uN}{\vphp_x}
+\sum_{j=1}^{N}\sum_{i=1}^{N-1}\dual{\uN^{-}_{i,y}}{\jump{\vphp}_{i,y}}_{J_j}
+\dual{ \uN}{\vphq_y}
+\sum_{i=1}^{N}\sum_{j=1}^{N-1}\dual{\uN^{-}_{x,j}}{\jump{\vphq}_{x,j}}_{I_i},
\nonumber\\
\mathcal{T}_3(\wN;\vphu)&=
\dual{\pN}{\vphu_x}
+\sum_{j=1}^{N}\left[
\sum_{i=0}^{N-1}\dual{\pN^{+}_{i,y}}{\jump{\vphu}_{i,y}}_{J_j}
-\dual{\pN^{-}_{N,y}}{\vphu^{-}_{N,y}}_{J_j}
\right]
\nonumber\\
&\qquad+\dual{\qN}{\vphu_y}
+\sum_{i=1}^{N}\left[
\sum_{j=0}^{N-1}\dual{\qN^{+}_{x,j}}{\jump{\vphu}_{x,j}}_{I_i}
-\dual{\qN^{-}_{x,N}}{\vphu^{-}_{x,N}}_{I_i}
\right],
\nonumber\\
\mathcal{T}_4(\uN;\vphu)&=
-\dual{a\uN}{\vphu_x}
-\sum_{j=1}^{N}\sum_{i=1}^{N}
\dual{a_{i,y}\uN^{-}_{i,y}}{\jump{\vphu}_{i,y}}_{J_j}
+\sum_{j=1}^{N}\dual{\lambda_{1}\uN^{-}_{N,y}}{\vphu^{-}_{N,y}}_{J_j} \notag\\
&\qquad+\sum_{i=1}^{N}\dual{\lambda_{2}\uN^{-}_{x,N}}{\vphu^{-}_{x,N}}_{I_i}
.
\nonumber
\end{align}
In this definition we choose the penalty parameters $\lambda_{1}$ 
and $\lambda_{2}$ to satisfy $0\leq \lambda_{1} \leq C$ and 
$C_1\varepsilon \leq \lambda_{2} \leq C$
for some constants $C_1$ and $C$;  
these parameters improve the stability and accuracy of the numerical scheme.
An explanation for the restriction $\lambda_{2}\ge C_1\varepsilon $
will be given later, following~\eqref{T3}.

Define the energy norm $\enorm{\cdot}$ 
by $\enorm{\vph}^2=B(\vph;\vph)$ for each $\vph = (\vphu,\vphp,\vphq)\in \spc^3$; that is, 
\begin{align*}
\enorm{\vph}^2
&=\lnorm{\vph}^2
+\sum_{j=1}^{N}
\sum_{i=0}^{N}\frac12 \dual{a_{i,y}}{\jump{\vphu}^2_{i,y}}_{J_j}
+\sum_{j=1}^{N}
\dual{\lambda_{1}}{\jump{\vphu}^2_{N,y}}_{J_j}
+\sum_{i=1}^{N}
\dual{\lambda_{2}}{\jump{\vphu}^2_{x,N}}_{I_i}
,
\\
&\text{where } \lnorm{\vph}^2
:=
\varepsilon^{-1}\norm{\vphp}^2+\varepsilon^{-1}\norm{\vphq}^2 
+ \norm{ \left(b- \frac12 a_x \right)^{1/2}\vphu}^2.\nonumber
\end{align*}
The linear system of equations \eqref{compact:form:2d} has a unique solution  
$\wN$ because the associated homogeneous problem (i.e., with $f=0$) has $\enorm{\wN} =0$ 
and hence $\wN = (0,0,0)$.

\section{The Gauss-Radau projection and its associated error}
\label{sec:GRproj}

\subsection{Definition and properties of the Gauss-Radau projection}
We define three two-dimensional local Gauss-Radau projectors $\Pi^{-}, \Pi_x^{+}, \Pi_y^{+}$ into $\spc$. 
For each $z\in H^2(\Omega_N)$,  
define $\Pi^{-}z \in \spc$ by the conditions:
\begin{align*}
\int_{K_{ij}}(\Pi^{-}z) \vphu \,\textrm{d}x\,\textrm{d}y
&= \int_{K_{ij}}z \vphu \,\textrm{d}x\,\textrm{d}y
\quad \forall \vphu\in \mathcal{Q}^{k-1}(K_{ij}),
\\
\int_{J_j}(\Pi^{-} z)_{i,y}^{-}\vphu\,\textrm{d}y
&=\int_{J_j} z_{i,y}^{-}\vphu\,\textrm{d}y
\qquad \forall \vphu\in \mathcal{P}^{k-1}(J_j),
\\
\int_{I_i}(\Pi^{-} z)_{x,j}^{-}\vphu\,\textrm{d}x
&=\int_{I_i} z_{x,j}^{-}\vphu\,\textrm{d}x
\qquad \forall \vphu\in \mathcal{P}^{k-1}(I_i),
\\
(\Pi^{-} z)(x_{i}^{-},y_{j}^{-}) &=	 z(x_{i}^{-},y_{j}^{-})
\end{align*}
for all elements $K_{ij}=I_i\times J_j=(x_{i-1},x_i)\times (y_{j-1},y_j)$ in $\Omega_{N}$,
where $z_{i,y}^{-}$ and $z_{x,j}^{-}$ are the edge traces
defined in Section~\ref{sec:LDG}.

For each $z\in H^1(\Omega_N)$,  define $\Pi_x^{+}z \in \spc$ by
\begin{subequations}\label{GR:p}
\begin{align}
\int_{K_{ij}}(\Pi_x^{+} z) \vphu
\,\textrm{d}x\,\textrm{d}y
&=
\int_{K_{ij}}z \vphu \,\textrm{d}x\,\textrm{d}y
\quad \forall \vphu\in \mathcal{P}^{k-1}(I_i)\otimes \mathcal{P}^k(J_j),
\\
\int_{J_j}(\Pi_x^{+} z)_{i-1,y}^{+}\vphu\,\textrm{d}y
&=
\int_{J_j}  z_{i-1,y}^{+}\vphu\,\textrm{d}y
\qquad \forall \vphu\in \mathcal{P}^k(J_j)
\end{align}
\end{subequations}
for all $K_{ij}\in\Omega_N$.
Analogously, for each $z\in H^1(\Omega_N)$,
define  $\Pi_y^{+}z \in \spc$ by
\begin{subequations}\label{GR:q}
\begin{align}
\int_{K_{ij}}(\Pi_y^{+}  z) \vphu\,\textrm{d}x\,\textrm{d}y
&=
\int_{K_{ij}}z \vphu \,\textrm{d}x\,\textrm{d}y
\quad \forall \vphu\in \mathcal{P}^{k}(I_i)\otimes \mathcal{P}^{k-1}(J_j),
\\
\int_{I_i}(\Pi_y^{+} z)^{+}_{x,j-1}\vphu\,\textrm{d}x
&=
\int_{I_i}  z^{+}_{x,j-1}\vphu\,\textrm{d}x
\qquad \forall \vphu\in \mathcal{P}^k(I_i)
\end{align}
\end{subequations}
for all $K_{ij}\in\Omega_N$.

Then
\begin{equation}\label{GR:tensor:product}
\Pi^{-}=\pi_x^{-}\otimes \pi_y^{-},
\quad
\Pi_x^{+}=\pi_x^{+}\otimes \pi_y,
\quad
\Pi_y^{+}=\pi_x\otimes \pi_y^{+},
\end{equation}
where $\pi$, $\pi_x^{-}$ and $\pi_y^{-}$
are the one-dimensional local $L^2$ projector and
the one-dimensional Gauss-Radau projectors  
in the $x$- and $y$-directions respectively
that are defined in~\cite[Section~3.1]{Castillo2002}.

Let  $\Pi\in\{\Pi^{-},\Pi_x^{+},\Pi_y^{+}\}$.
The following stability inequalities on each $K_{ij}$
can be deduced directly from the above definitions
(see, e.g.,  \cite[proof of Lemma 5]{Cheng2021:Calcolo}):
\begin{subequations}\label{2GR:stb}
	\begin{align}
	\label{GR:stb}
	\norm{\Pi z}_{L^\infty(K_{ij})}
	&\leq  C \norm{z}_{L^\infty(K_{ij})},
\\
	\label{GR:stb:L2:p}
	\norm{\Pi^+_x z}_{K_{ij}}&\leq  C
	\Big[
	\norm{z}_{K_{ij}} 
	+h^{1/2}_{x,i}\norm{z^+_{i-1,y}}_{J_j}
	\Big],
	\\
	\label{GR:stb:L2:q}
	\norm{\Pi^+_y z}_{K_{ij}}&\leq  C
	\Big[
	\norm{z}_{K_{ij}} + h^{1/2}_{y,j}\norm{z^+_{x,j-1}}_{I_i}
	\Big].
	\end{align}
\end{subequations}
One also has the anisotropic approximation property 
(\cite[Lemma 3]{Cheng2021:Calcolo}, \cite[Lemma 4.3]{Zhu:2dMC}):
	\begin{align}\label{2GR:app}
	\norm{z-\Pi z}_{L^{\ell}(K_{ij})}
	&\leq 
	C \left[h_{x,i}^{k+1}\norm{\partial_x^{k+1}z}_{L^{\ell}(K_{ij})}
	+h_{y,j}^{k+1}\norm{\partial_y^{k+1}z}_{L^{\ell}(K_{ij})}
	\right]\quad \text{for }\ell=2,\infty.
	\end{align}

%
%
\subsection{Gauss-Radau projection error of the true solution}\label{sec:GRtrue}

Set $\bm w := (u,p,q)$ and $\bm\Pi \bm w := (\prou u,\prop p,\proq q)$. The error in the Gauss-Radau projection of $\bm w$ is  
\[
\bm \eta  = (\eta_u,\eta_p,\eta_q) := (u-\prou u, p-\prop p, q-\proq q) = \bm w - \bm\Pi \bm w.
\]

To estimate $\bm\eta$ we make the following assumptions.
\begin{assumption}\label{ass:2}
	\mbox{ }
	\begin{itemize}
		\item[(i)]
		Assumption~\ref{assumption:reg:2d} is valid for $m=k$.
		\item[(ii)]
		Choose $\sigma\geq k+1$ in \eqref{tau}.
	\end{itemize}
\end{assumption}

In the next two lemmas we derive bounds on the components of $\bm\eta$.

\begin{lemma}\label{lemma:GR1}[Bounds on $\eta_u$]
There exists a constant $C>0$ such that
	\begin{subequations}\label{2dGR:property1}
		\begin{align}
		\label{etau:L2}
        \norm{\eta_u}
        &\leq C\left[N^{-(k+1)}+ \sq^{1/4} (N^{-1}\ln N)^{k+1}\right],
		\\
		\label{etau:L2:layer}
		\norm{\eta_u}_{\Omega_{21}\cup \Omega_{22}}
		&\leq 
		C\sq^{1/2} (N^{-1}\ln N)^{k+1},
		\\
		\label{etau:bry:left:x}
		\left(\sum_{j=1}^{N}\norm{(\eta_u)_{i,y}^{-}}^2_{J_j}\right)^{1/2}
		&\leq
		C\left[N^{-(k+1)}+\sq^{1/4}(N^{-1}\ln N)^{k+1}\right]
		\text{ for } i=1,...,N,
		\\
		\label{etau:bry:right}
		\left(\sum_{i=1}^{N}\norm{(\eta_u)_{x,N}^{-}}^2_{I_i}\right)^{1/2}
		&\leq  C\Big[N^{-(k+1)}+\sq^{1/2}(N^{-1}\ln N)^{k+1}\Big],
		\\
		\label{etau:bry:left}
		\left(\sum_{i=1}^{N}\sum_{j=1}^{N}
		\norm{(\eta_u)_{i,y}^{-}}_{J_j}^2\right)^{1/2}
		&\leq  
		C(N^{-1}\ln N)^{k+1/2},
		\\
		\label{etau:jump}
		\left(\sum_{j=1}^{N}\sum_{i=0}^{N}
		\dual{1}{\jump{\eta_u}^2_{i,y}}_{J_j}\right)^{1/2}
		&\leq  C(N^{-1}\ln N)^{k+1/2}.
		\end{align}
	\end{subequations}
\end{lemma}

\begin{proof}

Recall the decomposition of~$u$ in Assumption~\ref{assumption:reg:2d}. 
For each component~$u_z$ of~$u$, set
$\eta_{u_z}:= u_z - \Pi^-u_z$.
Then $\eta_u = \eta_{u_0} +\eta_{u_{1}}+\eta_{u_{2}}+ \eta_{u_{12}}$.

\emph{Proof of \eqref{etau:L2} and \eqref{etau:L2:layer}:}  By using \eqref{2GR:app}, \eqref{reg:u0} 
and the measure of each subregion $\Omega_{ij}$, one obtains easily 
\begin{align*}
\norm{\eta_{u_0}}_{\Omega_{11}}&\leq CN^{-(k+1)},
\hspace{2.5cm}
\norm{\eta_{u_0}}_{\Omega_{12}}\leq C\sq^{1/4} N^{-(k+1)} (\ln N)^{1/2} ,
\\
\norm{\eta_{u_0}}_{\Omega_{21}}&\leq C\sq^{1/2} N^{-(k+1)} (\ln N)^{1/2},
\quad
\norm{\eta_{u_0}}_{\Omega_{22}}\leq C\sq^{3/4} N^{-(k+1)} \ln N .
\end{align*}

For the exponential layer component $u_1$,
 the $L^{\infty}$-stability property \eqref{GR:stb} and \eqref{reg:u1}
yield
\begin{align}\label{eta:u1:Linf}
\norm{\eta_{u_1}}_{L^{\infty}(K_{ij})}
&\leq C \norm{u_1}_{L^{\infty}(K_{ij})}
\leq C e^{-\alpha(1-x_{i})/\sq}
\leq C e^{-\alpha\tau_1/\sq}
\leq C N^{-\sigma}
\end{align}
for each $K_{ij}\in \Omega_{11}\cup \Omega_{12}$. Hence
\[
\norm{\eta_{u_{1}}}_{\Omega_{11}\cup \Omega_{12}}
\leq \norm{\eta_{u_1}}_{L^{\infty}(\Omega_{11}\cup \Omega_{12})}
\leq C N^{-\sigma}.
\]
If $K_{ij}\in \Omega_{21}\cup \Omega_{22}$,
from \eqref{2GR:app} with $\ell=2$ and \eqref{reg:u1} one has
\begin{align*}
\norm{\eta_{u_{1}}}^2_{\Omega_{21}\cup \Omega_{22}}
&
\leq C\sum_{K_{ij}\in \Omega_{21}\cup \Omega_{22}}
\left[
h_{x,i}^{2(k+1)}\norm{\partial_x^{k+1}u_{1}}^2_{K_{ij}}
+h_{y,j}^{2(k+1)}\norm{\partial_y^{k+1}u_{1}}^2_{K_{ij}}
\right]
\\
&\leq C(N^{-1}\ln N)^{2(k+1)}\sum_{K_{ij}\in \Omega_{21}\cup \Omega_{22}}
\norm{e^{-\alpha(1-x)/\sq}}^2_{K_{ij}}
\leq C \sq (N^{-1}\ln N)^{2(k+1)}.\nonumber
\end{align*}

For the characteristic layer component $u_2$,
use \eqref{GR:stb} and \eqref{reg:u2}
to get 
\begin{align*}
\norm{\eta_{u_2}}_{L^{\infty}(K_{ij})}
&\leq C \norm{u_2}_{L^{\infty}(K_{ij})}
\leq C\left(e^{-\delta y_{j-1}/\ssq}+e^{-\delta(1-y_{j})/\ssq}\right)
\leq Ce^{-\delta \tau_2/\ssq}
\leq C N^{-\sigma}
\end{align*}
for each $K_{ij}\in \Omega_{11}\cup \Omega_{21}$. This bound implies that
\[
\norm{\eta_{u_{2}}}_{\Omega_{11}}\leq C N^{-\sigma}
\text{ and }
\norm{\eta_{u_{2}}}_{\Omega_{21}}
\leq C \tau_1^{1/2} N^{-\sigma}
\leq C \sq^{1/2} N^{-\sigma} (\ln N)^{1/2}.
\]
If $K_{ij}\in \Omega_{12}\cup \Omega_{22}$,
we use \eqref{2GR:app} with $\ell=2$ and \eqref{reg:u2}, obtaining
\begin{align*}
\norm{\eta_{u_{2}}}^2_{\Omega_{12}}
&
\leq C\sum_{K_{ij}\in \Omega_{12}}
\left[
h_{x,i}^{2(k+1)}\norm{\partial_x^{k+1}u_{2}}^2_{K_{ij}}
+h_{y,j}^{2(k+1)}\norm{\partial_y^{k+1}u_{2}}^2_{K_{ij}}
\right]
\nonumber\\
&\leq C(N^{-1}\ln N)^{2(k+1)}\sum_{K_{ij}\in \Omega_{12}}
\norm{e^{-\delta y/\ssq}+e^{-\delta(1-y)/\ssq}}^2_{K_{ij}}
\nonumber\\
&
\leq C \sq^{1/2} (N^{-1}\ln N)^{2(k+1)}.
\end{align*}
Similarly, 
\begin{align*}
\norm{\eta_{u_{2}}}^2_{\Omega_{22}}
&\leq C(N^{-1}\ln N)^{2(k+1)}\sum_{K_{ij}\in \Omega_{22}}
\norm{e^{-\delta y/\ssq}+e^{-\delta(1-y)/\ssq}}^2_{K_{ij}}
\nonumber\\
&
\leq C \tau_1 \sq^{1/2} (N^{-1}\ln N)^{2(k+1)}
=C\sq^{3/2} N^{-2(k+1)}(\ln N)^{2k+3}.
\end{align*}

For the corner layer component $u_{12}$,
the $L^{\infty}$-stability bound \eqref{GR:stb} and \eqref{reg:u2} 
give
\begin{align*}
\norm{\eta_{u_{12}}}_{L^{\infty}(K_{ij})}
&\leq C \norm{u_{12}}_{L^{\infty}(K_{ij})}
\leq Ce^{-\alpha(1-x_{i})/\sq}
\left(e^{-\delta y_{j-1}/\ssq}+e^{-\delta(1-y_{j})/\ssq}\right)
\end{align*}
for each $K_{ij}\in \Omega_{11}\cup \Omega_{12}\cup \Omega_{21}$, 
which implies
\begin{align*}
\norm{\eta_{u_{12}}}_{\Omega_{11}}
\leq C N^{-2\sigma},
\quad
\norm{\eta_{u_{12}}}_{\Omega_{12}}
\leq C \tau_2^{1/2} N^{-\sigma},
\quad
\norm{\eta_{u_{12}}}_{\Omega_{21}}
\leq C \tau_1^{1/2} N^{-\sigma}.
\end{align*}
If $K_{ij}\in \Omega_{22}$,
then using \eqref{2GR:app} with $\ell=2$ and \eqref{reg:u12} we get
\begin{align*}
\norm{\eta_{u_{12}}}^2_{\Omega_{22}}
&
\leq C\sum_{K_{ij}\in \Omega_{22}}
\left[
h_{x,i}^{2(k+1)}\norm{\partial_x^{k+1}u_{12}}^2_{K_{ij}}
+h_{y,j}^{2(k+1)}\norm{\partial_y^{k+1}u_{12}}^2_{K_{ij}}
\right]
\nonumber\\
&\leq C(N^{-1}\ln N)^{2(k+1)}
\sum_{K_{ij}\in \Omega_{22}}
\norm{e^{-\alpha(1-x)/\sq}
	\left(e^{-\delta y/\ssq}+e^{-\delta(1-y)/\ssq}\right)}^2_{K_{ij}}
\nonumber\\
&
\leq C \sq^{3/2} (N^{-1}\ln N)^{2(k+1)}.
\end{align*}

But $\tau_1=O(\sq\ln N)$,
$\tau_2=O(\ssq\ln N)\leq C$ 
and $\sigma\geq k+1$, 
so one can combine the above estimates to get 
 \eqref{etau:L2} and \eqref{etau:L2:layer}.

\emph{Proof of \eqref{etau:bry:left:x} and \eqref{etau:bry:right}:} 
Fix $i\in\{1,\dots, N\}$.
Recalling \eqref{GR:tensor:product},
one has
$(\eta_u)_{i,y}^{-}=u_{i,y}- \pi_y^{-}u_{i,y}$.
Assumption~\ref{assumption:reg:2d} implies that 
$u_{i,y}(y)=S_{i,y}(y)+E_{i,y}(y)$ 
where 
\[
\left|\frac{\mathrm{d}^j S_{i,y}(y)}{\mathrm{d} y^j}\right|\leq C
\text{ and } 
\left|\frac{\mathrm{d}^j E_{i,y}(y)}{\mathrm{d} y^j}\right|\leq 
C\varepsilon^{-j/2} \left(e^{-\delta y/\ssq}
+e^{-\delta(1-y)/\ssq}\right)
\ \text{ for }0\leq j \leq k+1.
\]
Set $\eta_{S_{i,y}}:=S_{i,y}- \pi_y^{-}S_{i,y}$ 
	and $\eta_{E_{i,y}}:=E_{i,y}- \pi_y^{-}E_{i,y}$.
By calculations similar to those used earlier in the proof, we get
\begin{align*}
\sum_{j=1}^{N}\norm{\eta_{S_{i,y}}}^2_{J_j}&\leq CN^{-2(k+1)},
\quad
\sum_{j=N/4+1}^{3N/4}\norm{\eta_{E_{i,y}}}^2_{J_j}\leq CN^{-2\sigma},
\\
\sum_{j=1}^{N/4}\norm{\eta_{E_{i,y}}}^2_{J_j}
\leq C \sum_{j=1}^{N/4} 
&(N^{-1}\ln N)^{2(k+1)}
\norm{e^{-\delta y/\ssq}+e^{-\delta(1-y)/\ssq}}^2_{J_{j}}
\leq C\sq^{1/2} (N^{-1}\ln N)^{2(k+1)},
\end{align*}
and likewise for the term $\sum_{j=3N/4+1}^{N}\norm{\eta_{E_{i,y}}}^2_{J_j}$.
Then  \eqref{etau:bry:left:x} follows from the above estimates
and a triangle inequality.

One can prove \eqref{etau:bry:right} by a similar argument.

\emph{Proof of \eqref{etau:bry:left}:} It follows from \eqref{etau:bry:left:x} 
and $\ssq\ln N\leq C$ that
\begin{align*}
\sum_{i=1}^{N}\sum_{j=1}^{N}\norm{(\eta_u)_{i,y}^{-}}^2_{J_j}
\leq CN\Big[
N^{-2(k+1)}+\sq^{1/2}(N^{-1}\ln N)^{2(k+1)}
\Big]
\leq C(N^{-1}\ln N)^{2k+1}.
\end{align*}

\emph{Proof of \eqref{etau:jump}:}
 Using \eqref{2GR:app} with $\ell=\infty$
and \eqref{reg:u0}, we have
\begin{align}\label{L0:u0}
\sum_{j=1}^{N}\sum_{i=1}^{N} h_{y,j}
\norm{\eta_{u_0}}^2_{L^{\infty}(K_{ij})}
\leq C\sum_{j=1}^{N}\sum_{i=1}^{N} h_{y,j} N^{-2(k+1)}
\leq CN^{-(2k+1)}.
\end{align}
Next, by \eqref{eta:u1:Linf} and the $L^{\infty}$-approximation property
\eqref{2GR:app} with $\ell=\infty$ we see that
\[
\norm{\eta_{u_1}}_{L^{\infty}(K_{ij})}
\leq \begin{cases}
	C N^{-\sigma} &\text{if } K_{ij}\in \Omega_{11}\cup \Omega_{12}, \\
	C(N^{-1}\ln N)^{k+1} e^{-\alpha(1-x_i)/\sq} &\text{if } K_{ij}\in \Omega_{21}\cup \Omega_{22}.
	\end{cases}
\]
Hence
\begin{align}
\sum_{j=1}^{N}\sum_{i=1}^{N} h_{y,j}
\norm{\eta_{u_1}}^2_{L^{\infty}(K_{ij})}
&=\sum_{j=1}^{N}\sum_{i=1}^{N/2} h_{y,j}
\norm{\eta_{u_1}}^2_{L^{\infty}(K_{ij})}
+\sum_{j=1}^{N}\sum_{i=N/2+1}^{N} h_{y,j}
\norm{\eta_{u_1}}^2_{L^{\infty}(K_{ij})}
\nonumber\\
&\leq C\sum_{j=1}^{N}\sum_{i=1}^{N/2} h_{y,j} N^{-2\sigma}
+ C \sum_{i=N/2+1}^{N} (N^{-1}\ln N)^{2(k+1)} e^{-2\alpha(1-x_i)/\sq}
\nonumber\\
&\leq
CN^{-2\sigma+1}+ C(N^{-1}\ln N)^{2k+1},
\label{L0:u1}
\end{align}
since $\sum_{j=1}^{N}h_{y,j}=1$, the definition \eqref{mesh:size:x:y} gives 
$h_{x,N/2+1} =\dots=h_{x,N}= O(\sq N^{-1}\ln N)$ so 
$\sum_{i=N/2+1}^N e^{-2\alpha(1-x_i)/\sq}$ is a geometric series,
and the well-known formula for the sum of such a series yields
\[
\sum_{i=N/2+1}^N e^{-2\alpha(1-x_i)/\sq}
\leq \frac1{1-e^{-2\alpha h_{x,N}/\sq}}
\leq C(N^{-1}\ln N)^{-1}
\]
because   $1-e^{-2\alpha h_{x,N}/\sq} = O(N^{-1}\ln N)$.
Similarly, one has
\begin{align*}
\sum_{i=1}^{N}\sum_{j=N/4+1}^{3N/4} h_{y,j}
\norm{\eta_{u_2}}^2_{L^{\infty}(K_{ij})}
&\leq CN^{-2\sigma+1},
\nonumber\\
\sum_{i=1}^{N}\sum_{j=1}^{N/4} h_{y,j}
\norm{\eta_{u_2}}^2_{L^{\infty}(K_{ij})}
&\leq 
C\sum_{i=1}^{N}\sum_{j=1}^{N/4} h_{y,j}
(N^{-1}\ln N)^{2(k+1)}
\left(e^{-2\delta y_{j-1}/\ssq}+e^{-2\delta(1-y_{j})/\ssq}\right)
\nonumber\\
&\leq C\tau_2(N^{-1}\ln N)^{2k+1},
\nonumber\\
\intertext{and}
\sum_{i=1}^{N}\sum_{j=3N/4+1}^{N} h_{y,j}
\norm{\eta_{u_2}}^2_{L^{\infty}(K_{ij})}
&\leq C\tau_2(N^{-1}\ln N)^{2k+1}.
\end{align*}
From these bounds it follows that
\begin{align}\label{L0:u2}
\sum_{i=1}^{N}\sum_{j=1}^{N} h_{y,j}
\norm{\eta_{u_2}}^2_{L^{\infty}(K_{ij})}
\leq C\Big[N^{-2\sigma+1}+(\sq^{1/2}\ln N) (N^{-1}\ln N)^{2k+1} \Big].
\end{align}
A similar calculation yields 
\begin{align}
\label{L0:u12}
\sum_{i=1}^{N}\sum_{j=1}^{N} h_{y,j}
\norm{\eta_{u_{12}}}^2_{L^{\infty}(K_{ij})}
\leq C\Big[N^{-2\sigma+1}+\sq^{1/2}(N^{-1}\ln N)^{2k+1}\Big].
\end{align}
Now \eqref{etau:jump} follows from \eqref{L0:u0}--\eqref{L0:u12},
$\sigma\ge k+1$, $\ssq\ln N\leq C$ and
\begin{align*}
\sum_{j=1}^{N}\sum_{i=0}^{N}
\dual{1}{\jump{\eta_u}^2_{i,y}}_{J_j}
&\leq 2
\sum_{j=1}^{N}\sum_{i=1}^{N}
\left(\norm{(\eta_u)_{i-1,y}^{+}}_{J_j}^2
+\norm{(\eta_u)_{i,y}^{-}}_{J_j}^2
\right)
\leq C\sum_{j=1}^{N}\sum_{i=1}^{N} h_{y,j}
\norm{\eta_u}^2_{L^{\infty}(K_{ij})}.
\end{align*}
\end{proof}

\begin{remark}
The above proof of Lemma~\ref{lemma:GR1} remains valid if $m=k-1$ in Assumption~\ref{ass:2}, 
but to prove Lemma~\ref{lemma:GR2} we shall need $m=k$ in this assumption.
\end{remark}

\begin{lemma}\label{lemma:GR2}[Bounds on $\eta_p$ and $\eta_q$]
There exists a constant $C>0$ such that
	\begin{subequations}\label{2dGR:property2}
		\begin{align}
		\label{etap:L2}
		\varepsilon^{-1/2}\norm{\eta_p}
		&\leq  C (N^{-1}\ln N)^{k+1},
		\\
		\label{etap:bry:right}
		\left(\sum_{j=1}^{N}\norm{(\eta_p)_{N,y}^{-}}^2_{J_j}\right)^{1/2}
		&\leq  C(N^{-1}\ln N)^{k+1},
				\\
		\label{etaq:bry:top}
		\left(\sum_{i=1}^{N}\norm{(\eta_q)_{x,N}^{-}}^2_{I_i}\right)^{1/2}
		&\leq  C\sq^{1/2} (N^{-1}\ln N)^{k+1},
			\\
		\label{etaq:L2}
		\varepsilon^{-1/2}\norm{\eta_q}
		&\leq  C\big[N^{-\sigma}+\sq^{1/4}(N^{-1}\ln N)^{k+1}\big].
		\end{align}
	\end{subequations}
\end{lemma}

\begin{proof}\mbox{  }\\
\emph{Proof of \eqref{etap:L2}:}
From Assumption~\ref{ass:2} 
(i.e., Assumption~\ref{assumption:reg:2d} with $m=k$)
it follows that
$p=\sq u_x= \sq u_{0,x} + \sq u_{1,x}+ \sq u_{2,x} +\sq u_{12,x}
:= p_0 + p_1 + p_2 +p_{12}$, where
	\begin{subequations}\label{reg:p:2d}
	\begin{align}
	\label{reg:p0}
	\left| \partial_x^{i}\partial_y^{j}p_0(x,y)\right|
	\leq&\; C\sq,
	\\
	\label{reg:p1}
	\left| \partial_x^{i}\partial_y^{j}p_{1}(x,y)\right|
	\leq &\; C\varepsilon^{-i} e^{-\alpha(1-x)/\varepsilon},
	\\
	\label{reg:p2}
	\left| \partial_x^{i}\partial_y^{j}p_{2}(x,y)\right|
	\leq &\; C\varepsilon^{1-j/2} \left(e^{-\delta y/\ssq}
	+e^{-\delta(1-y)/\ssq}\right),
	\\
	\label{reg:p12}
	\left| \partial_x^{i}\partial_y^{j}p_{12}(x,y)\right|
	\leq &\; C\varepsilon^{-(i+j/2)}e^{-\alpha(1-x)/\varepsilon}
	\left(e^{-\delta y/\ssq}+e^{-\delta(1-y)/\ssq}\right)
	\end{align}
\end{subequations}
for all $(x,y)\in \bar{\Omega}$ and
all nonnegative integers $i,j$ with $i+j\le k+1$. 
Set $\eta_{p_z} :=p_z - \Pi_x^+ p_z$ 
for each $p_z\in\{p_0,p_1,p_2,p_{12}\}$.

It is easy to see from \eqref{2GR:app} and \eqref{reg:p0} that
$\norm{\eta_{p_0}}\leq C\sq N^{-(k+1)}$.

For $K_{ij}\in \Omega_{11}\cup \Omega_{12}$,
we use \eqref{reg:p1}  and the $L^{2}$-stability property \eqref{GR:stb:L2:p} 
to obtain
\begin{align*}
\norm{\eta_{p_1}}^2_{K_{ij}}
&\leq C \left(
\norm{p_1}^2_{K_{ij}}+h_{x,i}\norm{(p_1)^{+}_{i-1,y}}^2_{J_j}
\right)\\
&\leq C h_{y,j}\left(\norm{e^{-\alpha(1-x)/\sq}}^2_{I_i}
+ h_{x,i}e^{-2\alpha(1-x_{i-1})/\sq}
\right)
\leq C h_{y,j} \norm{e^{-\alpha(1-x)/\sq}}^2_{I_i},
\end{align*}
because $x\mapsto e^{-\alpha(1-x)/\sq}$ is a monotonically increasing function.
Hence
\begin{align*}
\norm{\eta_{p_1}}^2_{\Omega_{11}\cup \Omega_{12}}
\leq \sum_{j=1}^N \sum_{i=1}^{N/2}
h_{y,j} \norm{e^{-\alpha(1-x)/\sq}}^2_{I_i}
\leq C\sq e^{-2\alpha(1-x_{N/2})/\sq}
\leq C\sq N^{-2\sigma}.
\end{align*}
For $K_{ij}\in \Omega_{21}\cup \Omega_{22}$,
\eqref{reg:p1} and the $L^{2}$-approximation property \eqref{2GR:app} 
yield
\begin{align*}
\norm{\eta_{p_1}}^2_{K_{ij}}
&\leq C 
\left[
h_{x,i}^{2(k+1)}
\norm{\partial_x^{k+1}p_{1}}^2_{K_{ij}}
+h_{y,j}^{2(k+1)}\norm{\partial_y^{k+1}p_{1}}^2_{K_{ij}}
\right]
\\
&\leq
Ch_{y,j}(N^{-1}\ln N)^{2(k+1)}
\norm{e^{-\alpha(1-x)/\sq}}^2_{I_{i}},
\end{align*}
which leads to
\begin{align*}
\norm{\eta_{p_1}}^2_{\Omega_{21}\cup \Omega_{22}}
&\leq
C(N^{-1}\ln N)^{2(k+1)} \sum_{j=1}^{N} \sum_{i=N/2+1}^{N}h_{y,j}\norm{e^{-\alpha(1-x)/\sq}}^2_{I_i}
\leq C\sq (N^{-1}\ln N)^{2(k+1)}.
\end{align*}

In a similar manner, one uses \eqref{GR:stb:L2:p} and  \eqref{reg:p2}  to get
\begin{align*}
\norm{\eta_{p_2}}^2_{K_{ij}}
&\leq C \left(
\norm{p_2}^2_{K_{ij}}+h_{x,i}\norm{(p_2)^{+}_{i-1,y}}^2_{J_j}
\right)\\
&\leq C\sq^2 \left(\norm{e^{-\delta y/\ssq}+e^{-\delta(1-y)/\ssq}}^2_{K_{ij}}
+ h_{x,i}\norm{e^{-\delta y/\ssq}+e^{-\delta(1-y)/\ssq}}^2_{J_j}
\right)\\
&\leq C\sq^2 h_{x,i}\norm{e^{-\delta y/\ssq}+e^{-\delta(1-y)/\ssq}}^2_{J_j}
\end{align*} 
for each $K_{ij}\in \Omega_N$. One can hence obtain the bounds
\begin{align*}
\norm{\eta_{p_2}}^2_{\Omega_{11}\cup \Omega_{21}}
&\leq C\sq^2\sum_{i=1}^N \sum_{j=N/4+1}^{3N/4}
h_{x,i} \norm{e^{-\delta y/\ssq}+e^{-\delta(1-y)/\ssq}}^2_{J_{j}}
\leq C\sq^{5/2} N^{-2\sigma},
\\
\norm{\eta_{p_2}}^2_{\Omega_{12}\cup \Omega_{22}}
&\leq C\sq^2 \sum_{i=1}^N \left(\sum_{j=1}^{N/4}+\sum_{j=3N/4+1}^{N}\right)
h_{x,i}(N^{-1}\ln N)^{2(k+1)}
\norm{e^{-\delta y/\ssq}+e^{-\delta(1-y)/\ssq}}^2_{J_{j}}
\\
&\leq C\sq^{5/2} (N^{-1}\ln N)^{2(k+1)}.
\end{align*}

For the corner layer component $\eta_{p_{12}}$, one gets likewise
\begin{align*}
\norm{\eta_{p_{12}}}^2_{\Omega_{11}}
&\leq C \sq^{3/2} N^{-4\sigma},
\quad
\norm{\eta_{p_{12}}}^2_{\Omega_{12}}
\leq C \sq^{3/2}N^{-2\sigma},
\\
\norm{\eta_{p_{12}}}^2_{\Omega_{21}}
&\leq C \sq^{3/2} N^{-2\sigma}, 
\quad
\norm{\eta_{p_{12}}}^2_{\Omega_{22}}
\leq C\sq^{3/2} (N^{-1}\ln N)^{2(k+1)}.
\end{align*}
This completes the proof of \eqref{etap:L2}.

\emph{Proof of \eqref{etap:bry:right}:}
Using \eqref{reg:p:2d},
the $L^{\infty}$-stability property \eqref{GR:stb}
and the $L^{\infty}$-approximation property \eqref{2GR:app}
with $\ell=\infty$, for $j=1,\dots, N$ we get 
\begin{align*}
\norm{\eta_{p_0}}_{L^{\infty}(K_{Nj})}
&\leq C\sq N^{-(k+1)},
\qquad
\norm{\eta_{p_1}}_{L^{\infty}(K_{Nj})}
\leq C(N^{-1}\ln N)^{k+1},
\\
\norm{\eta_{p_2}}_{L^{\infty}(K_{Nj})}
&\leq C\sq N^{-\sigma}+C\sq(N^{-1}\ln N)^{k+1},
\qquad
\norm{\eta_{p_{12}}}_{L^{\infty}(K_{Nj})}
\leq C N^{-\sigma}+C(N^{-1}\ln N)^{k+1}.
\end{align*}
Then  \eqref{etap:bry:right} follows using $\sigma\ge k+1$.

\emph{Proof of \eqref{etaq:bry:top}:}
Define the decomposition 
$q=\sq u_y= \sq u_{0,y} + \sq u_{1,y}+ \sq u_{2,y} +\sq u_{12,y}
:= q_0 + q_1 + q_2 +q_{12}$, where
\begin{subequations}\label{reg:q}
	\begin{align}
	\label{reg:q0}
	\left| \partial_x^{i}\partial_y^{j}q_0(x,y)\right|
	\leq&\; C\sq,
	\\
	\label{reg:q1}
	\left| \partial_x^{i}\partial_y^{j}q_{1}(x,y)\right|
	\leq &\; C\varepsilon^{1-i} e^{-\alpha(1-x)/\varepsilon},
	\\
	\label{reg:q2}
	\left| \partial_x^{i}\partial_y^{j}q_{2}(x,y)\right|
	\leq &\; C\varepsilon^{1/2-j/2} (e^{-\delta y/\ssq}
	+e^{-\delta(1-y)/\ssq}),
	\\
	\label{reg:q12}
	\left| \partial_x^{i}\partial_y^{j}q_{12}(x,y)\right|
	\leq &\; C\varepsilon^{1/2-(i+j/2)}e^{-\alpha(1-x)/\varepsilon}
	(e^{-\delta y/\ssq}+e^{-\delta(1-y)/\ssq})
	\end{align}
\end{subequations}
for all $(x,y)\in \bar{\Omega}$ and
all nonnegative integers $i,j$ with $i+j\le k+1$. 
Set $\eta_{q_z} :=q_z - \Pi_y^+ q_z$ 
for each $q_z\in\{q_0,q_1,q_2,q_{12}\}$.

One proves \eqref{etaq:bry:top} 
similarly to \eqref{etap:bry:right}, replacing $p$ by $q$ everywhere 
and observing that \eqref{reg:q1}--\eqref{reg:q12} each gain a factor of at least $\sq^{1/2}$
compared with \eqref{reg:p1}--\eqref{reg:p12}.

\emph{Proof of \eqref{etaq:L2}:}
We again use the decomposition $q= q_0 + q_1 + q_2 +q_{12}$ and the bounds \eqref{reg:q}.
First, it is easy to get $\norm{\eta_{q_0}}\leq C\sq N^{-(k+1)}$
from \eqref{2GR:app} and \eqref{reg:q0}.

For $K_{ij}\in \Omega_{11}\cup \Omega_{12}$,
we use the $L^{2}$-stability bound \eqref{GR:stb:L2:q} to get
\begin{align*}
\norm{\eta_{q_1}}^2_{K_{ij}}
&\leq C \left(
\norm{q_1}^2_{K_{ij}}+h_{y,j}\norm{(q_1)^{+}_{x,j-1}}^2_{I_i}
\right)
\leq C\sq^2 h_{y,j} \norm{e^{-\alpha(1-x)/\sq}}^2_{I_i}.
\end{align*}
This implies
\begin{align*}
\norm{\eta_{q_1}}^2_{\Omega_{11}\cup \Omega_{12}}
\leq C\sq^2\sum_{j=1}^N \sum_{i=1}^{N/2}
h_{y,j} \norm{e^{-\alpha(1-x)/\sq}}^2_{I_i}
\leq C\sq^3 e^{-2\alpha(1-x_{N/2})/\sq}
\leq C\sq^3 N^{-2\sigma}.
\end{align*}
The $L^{2}$-approximation property \eqref{2GR:app} with $\ell=2$ and \eqref{reg:q1} yield
\begin{align*}
\norm{\eta_{q_1}}^2_{\Omega_{21}\cup \Omega_{22}}
&
\leq 
\sum_{K_{ij}\in \Omega_{21}\cup \Omega_{22}}
C\sq^2 (N^{-1}\ln N)^{2(k+1)}\norm{e^{-\alpha(1-x)/\sq}}^2_{K_{ij}}
\leq C\sq^3 (N^{-1}\ln N)^{2(k+1)}.
\end{align*}

For $\eta_{q_2}$ we are unable to imitate the  analysis of $\eta_{q_1}$ 
because the function $y\mapsto e^{-\delta y/\ssq}+e^{-\delta(1-y)/\ssq}$
is not monotone. Instead we invoke the $L^{\infty}$-stability 
bound \eqref{GR:stb}, getting
\begin{align*}
\norm{\eta_{q_2}}_{L^{\infty}(K_{ij})}
&\leq C \norm{q_2}_{L^{\infty}(K_{ij})}
\leq C \sq^{1/2} (e^{-\delta y_{j-1}/\ssq}+e^{-\delta(1-y_j)/\ssq})
\leq C \sq^{1/2} N^{-\sigma}
\end{align*}
for each $K_{ij}\in \Omega_{11}\cup \Omega_{21}$.
Hence
\[
\norm{\eta_{q_2}}_{\Omega_{11}\cup \Omega_{21}}\leq
\norm{\eta_{q_2}}_{L^{\infty}(\Omega_{11}\cup \Omega_{21})}
\leq C\sq^{1/2} N^{-\sigma}.
\]
The $L^2$-approximation property \eqref{2GR:app} with $\ell=2$ and \eqref{reg:q2} yield
\begin{align*}
\norm{\eta_{q_2}}^2_{\Omega_{12}\cup \Omega_{22}}
&\leq \sum_{K_{ij}\in \Omega_{12}\cup \Omega_{22}}
C\sq (N^{-1}\ln N)^{2(k+1)}
\norm{e^{-\delta y/\ssq}+e^{-\delta(1-y)/\ssq}}^2_{K_{ij}}
\\
&\leq C\sq^{3/2} (N^{-1}\ln N)^{2(k+1)}.
\end{align*}
Similarly, one has
\begin{align*}
\norm{\eta_{q_{12}}}^2_{\Omega_{11}\cup \Omega_{21}\cup \Omega_{12}}
\leq C\sq N^{-2\sigma},
\quad
\norm{\eta_{q_{12}}}^2_{\Omega_{22}}
\leq C\sq^{5/2} (N^{-1}\ln N)^{2(k+1)}.
\end{align*}
Adding these bounds, one gets \eqref{etaq:L2}.
\end{proof}

\subsection{A superapproximation result}\label{sec:superapprox}

For each element $K_{ij}\in \Omega_N$ and each $v\in \spc$,
define the two local bilinear forms 
\begin{subequations}\label{bilinear:D}
\begin{align}
{\mathcal D}^1_{ij}(\eta_u;v)
&:= \dual{\eta_u}{v_x}_{K_{ij}}
-\dual{(\eta_u)^{-}_{i,y}}{v^{-}_{i,y}}_{J_{j}}
+\dual{(\eta_u)^{-}_{i-1,y}}{v^{+}_{i-1,y}}_{J_{j}},
\\
{\mathcal D}^2_{ij}(\eta_u;v)
&:= \dual{\eta_u}{v_y}_{K_{ij}}
-\dual{(\eta_u)^{-}_{x,j}}{v^{-}_{x,j}}_{I_{i}}
+\dual{(\eta_u)^{-}_{x,j-1}}{v^{+}_{x,j-1}}_{I_{i}},
\end{align}
\end{subequations}
where we set $(\eta_u)^{-}_{0,y}=(\eta_u)^{-}_{x,0}=0$.

The next lemma presents a superapproximation result for these operators that we will need in our error analysis.
While $\sigma\geq k+1$ sufficed for our previous analysis, we shall need $\sigma\geq k+2$ here.

\begin{lemma}\label{superapproximation:element}
Choose $\sigma\geq k+2$.
Assume that $\sq^{1/4}\ln N\leq C_2$  for some fixed constant~$C_2$. 
Then there exists a constant $C>0$ such that 
\begin{align}
\label{sup:x:left}
\left[
\sup_{s\in \spc}
\sum_{K_{ij}\in \Omega_{11}\cup \Omega_{12}}  
\left(\frac{{\mathcal D}^1_{ij}(\eta_u;s)}{\norm{s}_{K_{ij}}}\right)^2
\right]^{1/2}
&\leq C(N^{-1}\ln N)^{k+1},
\\
\label{sup:x:right}
\left[
\sup_{s\in \spc}
\sum_{K_{ij}\in \Omega_{21}\cup \Omega_{22}}  
\left(\frac{{\mathcal D}^1_{ij}(\eta_u;s)}{\norm{s}_{K_{ij}}}\right)^2
\right]^{1/2}
&\leq C\varepsilon^{-1/2}(N^{-1}\ln N)^{k+1},
\\
\label{sup:y}
\left[
\sup_{s\in \spc}
\sum_{K_{ij}\in \Omega_{N}}  
\left(\frac{{\mathcal D}^2_{ij}(\eta_u;s)}{\norm{s}_{K_{ij}}}\right)^2
\right]^{1/2}
&\leq C\varepsilon^{-1/4}(N^{-1}\ln N)^{k+1}.
\end{align}
\end{lemma}

\begin{proof}
We shall use the following stability and approximation properties
(\cite[(4.15b)]{Cheng:Jiang:Stynes:2022}, \cite[Lemma 4.8]{Zhu:2dMC}):
\begin{subequations}
	\label{stb:L0}
		\begin{align}
		\label{stab:L0:bilinear}
		|\mathcal{D}^{1}_{ij}(\eta_z;s)|
		&\leq C\sqrt{\frac{h_{y,j}}{h_{x,i}}}
		\norm{z}_{L^\infty(K_{ij})}\norm{s}_{K_{ij}},
		\\
		|\mathcal{D}^{1}_{ij}(\eta_z;s)|
		\label{sup:L2}
		&\leq
   Ch_{x,i}^{-1}
		\left[h_{x,i}^{k+2}\norm{\partial_x^{k+2}z}_{K_{ij}}
		+h_{y,j}^{k+2}\norm{\partial_y^{k+2}z}_{K_{ij}}
		\right]\norm{s}_{K_{ij}}
		\end{align}
\end{subequations}
for any function $z\in H^{k+2}(\Omega)$, $\eta_z:= z- \Pi^-z$, all $s\in \spc$ and $1\le i,j\le N$.

For the smooth component $u_0$, we use \eqref{sup:L2} and Assumption~\ref{ass:2} to get
\[
	|\mathcal{D}^{1}_{ij}(\eta_{u_0};s)|
\leq
C \sqrt{\frac{h_{y,j}}{h_{x,i}}}
\left[h_{x,i}^{k+2}+h_{y,j}^{k+2}\right]
\norm{s}_{K_{ij}}
\leq
C \sqrt{\frac{h_{y,j}}{h_{x,i}}} \,N^{-(k+2)}\norm{s}_{K_{ij}}
\quad \forall K_{ij}\in \Omega_{N}.
\]
Substituting the mesh sizes of \eqref{mesh:size:x:y} into this inequality, 
we get
\begin{align}\label{sup:D1:u0}
\sum_{K_{ij}\in D}  
\Bigg(\frac{{\mathcal D}^1_{ij}(\eta_{u_0};s)}{\norm{s}_{K_{ij}}}\Bigg)^2
&\leq
\begin{cases}
CN^{-2(k+1)} &\text{if }  D=\Omega_{11},\\
C\tau_2N^{-2(k+1)} &\text{if }  D=\Omega_{12},\\
C\sq^{-1}N^{-2(k+1)} &\text{if }  D=\Omega_{21},\\
C\sq^{-1/2}N^{-2(k+1)} &\text{if }  D=\Omega_{22}.
\end{cases}
\end{align}
For the exponential layer component $u_1$, 
from \eqref{stb:L0} and \eqref{reg:u1} one has
\[
|\mathcal{D}^{1}_{ij}(\eta_{u_1};s)|
\leq
\begin{cases}
C\sqrt{\frac{h_{y,j}}{h_{x,i}}}
e^{-\alpha(1-x_i)/\sq}\norm{s}_{K_{ij}}
&\text{if } K_{ij}\in \Omega_{11}\cup \Omega_{12},
\\
Ch_{x,i}^{-1}
(N^{-1}\ln N)^{k+2}
\norm{e^{-\alpha(1-x)/\sq}}_{K_{ij}}
\norm{s}_{K_{ij}}
&\text{if }  K_{ij}\in \Omega_{21}\cup \Omega_{22}.
\end{cases}
\]
Again using the mesh sizes of \eqref{mesh:size:x:y}, we deduce that
\begin{align}\label{sup:D1:u1}
\sum_{K_{ij}\in D}  
\Bigg(\frac{{\mathcal D}^1_{ij}(\eta_{u_1};s)}{\norm{s}_{K_{ij}}}\Bigg)^2
&\leq
\begin{cases}
CN^{-2\sigma+2}
&\text{if }  D=\Omega_{11},\\
C\tau_2N^{-2\sigma+2} 
&\text{if }  D=\Omega_{12},\\
C\sq^{-1}(N^{-1}\ln N)^{2(k+1)}
&\text{if }  D=\Omega_{21},\\
C\sq^{-1/2}(\ln N) (N^{-1}\ln N)^{2(k+1)} 
&\text{if }  D=\Omega_{22}.
\end{cases}
\end{align}
For the characteristic boundary layer component, 
from \eqref{stb:L0} and \eqref{reg:u2} one has
\begin{align*}
&\hspace{-5mm}|\mathcal{D}^{1}_{ij}(\eta_{u_2};s)| \\
&\leq
\begin{cases}
C\sqrt{\frac{h_{y,j}}{h_{x,i}}}
\left(e^{-\delta y_{j-1}/\ssq}+e^{-\delta(1-y_{j})/\ssq}\right)
\norm{s}_{K_{ij}}
&\text{if }  K_{ij}\in \Omega_{11}\cup \Omega_{21},
\\
Ch_{x,i}^{-1}
(N^{-1}\ln N)^{k+2}
\norm{e^{-\delta y/\ssq}+e^{-\delta(1-y)/\ssq}}_{K_{ij}}
\norm{s}_{K_{ij}}
&\text{if }  K_{ij}\in \Omega_{12}\cup \Omega_{22},
\end{cases}
\end{align*}
whence
\begin{align}\label{sup:D1:u2}
\sum_{K_{ij}\in D}  
\Bigg(\frac{{\mathcal D}^1_{ij}(\eta_{u_2};s)}{\norm{s}_{K_{ij}}}\Bigg)^2
&\leq
\begin{cases}
CN^{-2\sigma+2} 
&\text{if }  D=\Omega_{11},\\
C\sq^{1/2} (\ln N)^2 (N^{-1}\ln N)^{2(k+1)} 
&\text{if }  D=\Omega_{12},\\
C\sq^{-1}N^{-2\sigma+2}
&\text{if }  D=\Omega_{21},\\
C\sq^{-1/2}(\ln N) (N^{-1}\ln N)^{2(k+1)} 
&\text{if }  D=\Omega_{22}.
\end{cases}
\end{align}
For the corner layer component, we have
\begin{align*}
|\mathcal{D}^{1}_{ij}(\eta_{u_{12}};s)| 
&\leq
C\sqrt{\frac{h_{y,j}}{h_{x,i}}}
e^{-\alpha(1-x_i)/\sq}
\big(e^{-\delta y_{j-1}/\ssq}+e^{-\delta(1-y_{j})/\ssq}\big)
\norm{s}_{K_{ij}} \\
&\qquad\qquad\text{if }  K_{ij}\in \Omega_{11}\cup \Omega_{21}\cup \Omega_{12}, \\
|\mathcal{D}^{1}_{ij}(\eta_{u_{12}};s)| 
&\leq Ch_{x,i}^{-1}
(N^{-1}\ln N)^{k+2}
\norm{e^{-\alpha(1-x)/\sq}
	(e^{-\delta y/\ssq}+e^{-\delta(1-y)/\ssq})}_{K_{ij}}
\norm{s}_{K_{ij}} \\
&\qquad\qquad \text{if }  K_{ij}\in \Omega_{22},
\end{align*}
which lead to
\begin{align}\label{sup:D1:u12}
\sum_{K_{ij}\in D}  
\Bigg(\frac{{\mathcal D}^1_{ij}(\eta_{u_{12}};s)}{\norm{s}_{K_{ij}}}\Bigg)^2
&\leq
\begin{cases}
CN^{-4\sigma+2} 
&\text{if }  D=\Omega_{11},\\
C\tau_2N^{-2\sigma+2}
&\text{if }  D=\Omega_{12},\\
C\sq^{-1}N^{-2\sigma+2} 
&\text{if }  D=\Omega_{21},\\
C\sq^{-1/2}(N^{-1}\ln N)^{2(k+1)} 
&\text{if }  D=\Omega_{22}.
\end{cases}
\end{align}
The bounds \eqref{sup:D1:u0}--\eqref{sup:D1:u12}
and the hypotheses $\sigma \ge k+2$ and $\sq^{1/4}\ln N\leq C_2$ yield~\eqref{sup:x:left}
and \eqref{sup:x:right}.

For ${\mathcal D}^2_{ij}$  the analysis proceeds 
in an analogous manner;
in layer regions the inverse factor $h^{-1}_{x,i}$ 
that appeared above when $N/2+1\leq i \leq N$
is replaced by the less severe inverse factor $h^{-1}_{y,j}$
when $1\leq j \leq N/4$ and $3N/4+1\leq j \leq N$, which yields some improvements in the bounds.
We list the conclusions below:
\begin{align*}
\sum_{K_{ij}\in D}  
\Bigg(\frac{{\mathcal D}^2_{ij}(\eta_{u_0};s)}{\norm{s}_{K_{ij}}}\Bigg)^2
&\leq
\begin{cases}
CN^{-2(k+1)} &\text{if }  D=\Omega_{11},\\
C\sq^{-1/2}N^{-2(k+1)} &\text{if }  D=\Omega_{12},\\
C\tau_1 N^{-2(k+1)} &\text{if }  D=\Omega_{21},\\
C\sq^{1/2}N^{-2(k+1)} &\text{if }  D=\Omega_{22},
\end{cases}
\\
\sum_{K_{ij}\in D}  
\Bigg(\frac{{\mathcal D}^2_{ij}(\eta_{u_1};s)}{\norm{s}_{K_{ij}}}\Bigg)^2
&\leq
\begin{cases}
CN^{-2\sigma+2} 
&\text{if }  D=\Omega_{11},\\
C\sq^{-1/2} N^{-2\sigma+2} 
&\text{if }  D=\Omega_{12},\\
C\sq (\ln N)^2(N^{-1}\ln N)^{2(k+1)}
&\text{if }  D=\Omega_{21},\\
C\sq^{1/2}(\ln N) (N^{-1}\ln N)^{2(k+1)}
&\text{if }  D=\Omega_{22},
\end{cases}
\\
\sum_{K_{ij}\in D}  
\Bigg(\frac{{\mathcal D}^2_{ij}(\eta_{u_2};s)}{\norm{s}_{K_{ij}}}\Bigg)^2
&\leq
\begin{cases}
CN^{-2\sigma+2} 
&\text{if }  D=\Omega_{11},\\
C\sq^{-1/2}  (N^{-1}\ln N)^{2(k+1)} 
&\text{if }  D=\Omega_{12},\\
C\sq N^{-2\sigma+2}\ln N 
&\text{if }  D=\Omega_{21},\\
C\sq^{1/2}(\ln N) (N^{-1}\ln N)^{2(k+1)}
&\text{if }  D=\Omega_{22},
\end{cases}
\\
\sum_{K_{ij}\in D}  
\Bigg(\frac{{\mathcal D}^2_{ij}(\eta_{u_{12}};s)}{\norm{s}_{K_{ij}}}\Bigg)^2
&\leq
\begin{cases}
CN^{-4\sigma+2} 
&\text{if }  D=\Omega_{11},\\
C\sq^{-1/2}N^{-2\sigma+2} 
&\text{if }  D=\Omega_{12},\\
C\sq N^{-2\sigma+2}\ln N 
&\text{if }  D=\Omega_{21},\\
C\sq^{1/2}(N^{-1}\ln N)^{2(k+1)}
&\text{if }  D=\Omega_{22}.
\end{cases}
\end{align*}
Combining these four inequalities with 
$\ssq\ln N\leq C$ and $\sigma\geq k+2$, one gets \eqref{sup:y}.
\end{proof}

\begin{remark}
The hypothesis	$\sq^{1/4}\ln N\leq C_2$ of Lemma~\ref{superapproximation:element} was also used in \cite[Theorem 4]{Zarin2005}. If one removes it from Lemma~\ref{superapproximation:element},
then one can still use the assumption of Section~\ref{subsec:layer:adapted:meshes} 
that  $\sq^{1/2}\ln N\leq C$, which yields a slightly weaker version 
of this lemma where an additional factor $(\ln N)^{1/2}$ appears in~\eqref{sup:x:left}.
\end{remark}

\section{Error analysis}
\label{sec:analysis}

The error in the LDG solution is 
$\bm e=(e_u,e_p,e_q) :=(u-\uN,p-\pN,q-\qN)$,  so
\begin{equation}\label{error:decomposition}
\bm e = \bm w - \bm W = (\bm w-\bm\Pi \bm w)-(\wN-\bm\Pi \bm w)  = \bm \eta-\bm \xi,
\end{equation}
where we define 
\[
\bm \xi = (\xi_u,\xi_p,\xi_q)
:= (\uN-\prou u,	\pN-\prop p, \qN-\proq q)
\in \sobv^3.
\]

The true solution $\bm w=(u,p,q)$ satisfies
the weak formulation \eqref{compact:form:2d}, 
so one has the Galerkin orthogonality property
$
B(\bm w-\bm W;\vph)=0 \ \forall \vph\in \spc^3.
$
Taking $\vph=\bm\xi$ in this equation and  recalling
 the definition of~$\enorm{\cdot}$ and \eqref{error:decomposition},
one gets
\begin{align}\label{error:equation}
\enorm{\bm\xi}^2 =
B(\bm \xi;\bm\xi)=B(\bm \eta;\bm\xi)
=\mathcal{T}_1(\bm \eta;\bm\xi)+\mathcal{T}_2(\eta_u;\bm\xi)
+\mathcal{T}_3(\bm \eta;\xi_u)+\mathcal{T}_4(\eta_u;\xi_u),
\end{align}
where the terms $\mathcal{T}_i\,(i=1,2,3,4)$ are defined 
as in~\eqref{B:def:2d}.
To estimate $\bm\xi$
we shall bound each  of these $\mathcal{T}_i$.

\subsection{Bounds on the term $\mathcal{T}_i$ for $i=1,2,3$}
\label{sec:T123}

First, using a  Cauchy-Schwarz inequality and the definition
of $\enorm{\cdot}$, one has
\begin{align}\label{T1}
\left|\mathcal{T}_1(\bm \eta;\bm\xi)\right|
&= \left|\varepsilon^{-1}[\dual{\eta_p}{\xi_p}+\dual{\eta_q}{\xi_q}]
+\dual{(b- a_x)\eta_u}{\xi_u}\right|
\nonumber\\
&\leq (\varepsilon^{-1/2}\norm{\eta_p}+\varepsilon^{-1/2}\norm{\eta_q}
+C\norm{\eta_u})\enorm{\bm\xi}.
\end{align}

Next, one can rewrite the term $\mathcal{T}_2$ as
\[
\mathcal{T}_2(\eta_u;\bm\xi)
=
\sum_{K_{ij}\in \Omega_{N}}  {\mathcal D}^1_{ij}(\eta_u;\xi_p)
+\sum_{K_{ij}\in \Omega_{N}}  {\mathcal D}^2_{ij}(\eta_u;\xi_q),
\]
where $ {\mathcal D}^1_{ij}$ and $ {\mathcal D}^2_{ij}$ were defined in~\eqref{bilinear:D}.
Since the terms $\varepsilon^{-1/2}\norm{\xi_p}$
and $\varepsilon^{-1/2}\norm{\xi_q}$ appear in $\enorm{\bm\xi}$, one gets easily 
\begin{align}\label{T2}
\left|\mathcal{T}_2(\eta_u;\bm\xi)\right|
&\leq
\sq^{1/2}\left[ \sum_{K_{ij}\in \Omega_{N}}  
\left(\frac{{\mathcal D}^1_{ij}(\eta_u;\xi_p)}{\norm{\xi_p}_{K_{ij}}}\right)^2
+\sum_{K_{ij}\in \Omega_{N}}  
\left(\frac{{\mathcal D}^2_{ij}(\eta_u;\xi_q)}{\norm{\xi_q}_{K_{ij}}}\right)^2
\right] ^{1/2}
\enorm{\bm\xi}.
\end{align}

In $\mathcal{T}_3$, the definitions \eqref{GR:p} and \eqref{GR:q}
of the Gauss-Radau projection imply that
\[
\mathcal{T}_3(\bm\eta;\xi_u)
=
-\sum_{j=1}^N\dual{(\eta_p)^{-}_{N,y}}{(\xi_u)^{-}_{N,y}}_{J_j}
-\sum_{i=1}^N\dual{(\eta_q)^{-}_{x,N}}{(\xi_u)^{-}_{x,N}}_{I_i}.
\]
Hence, using a 
Cauchy-Schwarz inequality and 
the form of the boundary jump term in the energy norm, 
one has
\begin{align}
\left|\mathcal{T}_3(\bm\eta;\xi_u)\right| &\leq
C\left(\sum_{j=1}^N\norm{(\eta_p)^{-}_{N,y}}^2_{J_j}\right)^{1/2}
\left(\sum_{j=1}^{N}
\dual{\lambda_{1}+\frac12 a_{N,y}}{\jump{\xi_u}^2_{N,y}}_{J_j}\right)^{1/2}
\nonumber\\
&\qquad
+\lambda_2^{-1/2}
\left(\sum_{i=1}^N\norm{(\eta_q)^{-}_{x,N}}^2_{I_i}\right)^{1/2}
\left(\sum_{i=1}^{N}
\dual{\lambda_{2}}{\jump{\xi_u}^2_{x,N}}_{I_i}\right)^{1/2}
\nonumber\\
\label{T3}
&\leq
C
\left[
\left(\sum_{j=1}^N\norm{(\eta_p)^{-}_{N,y}}^2_{J_j}\right)^{1/2}
+\lambda_2^{-1/2}
\left(\sum_{i=1}^N\norm{(\eta_q)^{-}_{x,N}}^2_{I_i}\right)^{1/2}
\right]
\enorm{\bm\xi}.
\end{align}
Note that here we need $\lambda_2>0$; 
we assumed $\lambda_2\geq C_1\sq$ in Section~\ref{sec:LDG}
with the intention of 
invoking~\eqref{etaq:bry:top} to handle the $\lambda_2$ term in~\eqref{T3}.

\subsection{Bounds on  the components of $\mathcal{T}_4(\eta_u;\xi_u)$}
The term $\mathcal{T}_4(\eta_u;\xi_u)$ is much more difficult to handle than $\mathcal{T}_1,\mathcal{T}_2$ or $\mathcal{T}_3$.
We imitate \cite[Section 4]{Cheng:Jiang:Stynes:2022} by decomposing   $\mathcal{T}_4$ into 5 components, viz.,
$
\mathcal{T}_4(\eta_u;\xi_u)
= \sum_{i=1}^5 \mathcal{T}_{4i}(\eta_u;\xi_u),
$
where we set $a_{ij}:=a(x_i,y_j)$ and define
\begin{align*}
\mathcal{T}_{41}(\eta_u;\xi_u) &:=
-\sum_{j=1}^N\sum_{i=1}^{N/2} a_{ij}\mathcal{D}^1_{ij}(\eta_u;\xi_u),
\\
\mathcal{T}_{42}(\eta_u;\xi_u) &:=
-\sum_{j=1}^N\sum_{i=1}^{N/2}
\left[\dual{(a-a_{ij})\eta_u}{(\xi_{u})_{x}}_{K_{ij}}
-\dual{(a_{i,y}-a_{ij})(\eta_u)^{-}_{i,y}}{(\xi_u)^{-}_{i,y}}_{J_{j}} \right.
\\
&\hspace{3cm}
\left.
+\dual{(a_{i-1,y}-a_{ij})(\eta_u)^{-}_{i-1,y}}{(\xi_u)^{+}_{i-1,y}}_{J_{j}}\right],
\\
\mathcal{T}_{43}(\eta_u;\xi_u) &:=
-\sum_{j=1}^N\sum_{i=N/2+1}^N
\dual{a\eta_u}{(\xi_{u})_{x}}_{K_{ij}}
-\sum_{j=1}^N\sum_{i=N/2+1}^N
\dual{a_{i,y}(\eta_u)^{-}_{i,y}}{\jump{\xi_u}_{i,y}}_{J_j},
\\
\mathcal{T}_{44}(\eta_u;\xi_u) &:=
-\sum_{j=1}^N a_{N/2,j}
\dual{(\eta_u^{-})_{N/2,y}}{(\xi_u^{+})_{N/2,y}}_{J_j},
\\
\mathcal{T}_{45}(\eta_u;\xi_u) &:=
\sum_{j=1}^{N}
\dual{\lambda_{1}(\eta_u)^{-}_{N,y}}{(\xi_u)^{-}_{N,y}}_{J_j}
+\sum_{i=1}^{N}
\dual{\lambda_{2}(\eta_u)^{-}_{x,N}}{(\xi_u)^{-}_{x,N}}_{I_i}.
\end{align*}
Note that in $\mathcal{T}_{42}$ we use $(\eta_u)^{-}_{0,y}=0$,
as defined in Section~\ref{sec:superapprox}.

By a Cauchy-Schwarz inequality we get
\begin{align}\label{T41}
\left|\mathcal{T}_{41}(\eta_u;\xi_u)\right|
&\leq
C\left[\sum_{K_{ij}\in \Omega_{11}\cup \Omega_{12}}
\left(\frac{{\mathcal D}^1_{ij}(\eta_u;\xi_u)}{\norm{\xi_u}_{K_{ij}}}\right)^2
\right]^{1/2}
\enorm{\bm\xi}.
\end{align}

From \cite[Theorem 4.76]{Schwab1998} and a scaling argument, one has
the following anisotropic inverse and trace inequalities:
\[
\norm{v_x}_{K_{ij}}\leq Ch_{x,i}^{-1}\norm{v}_{K_{ij}}
\ \text{ and }\ 
\norm{v^{-}_{i,y}}_{J_{j}} +
\norm{v^{+}_{i-1,y}}_{J_{j}}\leq Ch_{x,i}^{-1/2}\norm{v}_{K_{ij}}
\ \text{ for all } v\in \spc,
\]
where $C>0$ is independent of $v$ and of the mesh element $K_{ij}$.
Hence a Cauchy-Schwarz inequality and 
$a-a_{ij}=O(N^{-1})$ for $(x,y)\in K_{ij}$
yield
\begin{align}\label{T42}
\left|\mathcal{T}_{42}(\eta_u;\xi_u)\right|
&\leq
C\sum_{j=1}^N\sum_{i=1}^{N/2} N^{-1}
\left[h^{-1}_{x,i}\norm{\eta_u}_{K_{ij}}
+h^{-1/2}_{x,i}
\left(\norm{(\eta_u)^{-}_{i-1,y}}_{J_{j}} + \norm{(\eta_u)^{-}_{i,y}}_{J_{j}}\right)
\right]\norm{\xi_u}_{K_{ij}}
\nonumber\\
&\leq C\left[\sum_{j=1}^N\sum_{i=1}^{N/2}
\left(\norm{\eta_u}^2_{K_{ij}}
+N^{-1}\norm{(\eta_u)^{-}_{i,y}}^2_{J_{j}}
\right)\right]^{1/2}\enorm{\bm\xi}.
\end{align}

The term $\mathcal{T}_{43}(\eta_{u};\xi_u)$ is bounded by
\begin{align}
\left|\mathcal{T}_{43}(\eta_{u};\xi_u)\right| 
&\leq 
C\norm{\eta_u}_{\Omega_{21}\cup \Omega_{22}}\norm{(\xi_u)_x}_{\Omega_{21}\cup \Omega_{22}}
+C\left(\sum_{j=1}^{N}
\norm{(\eta_u)_{N,y}^{-}}_{J_j}^2
\right)^{1/2}
\left(\sum_{j=1}^{N}
\norm{\jump{\xi_u}_{N,y}}_{J_j}^2\right)^{1/2}
\nonumber\\
&\qquad
+C
\left(\sum_{j=1}^{N}\sum_{i=N/2+1}^{N-1}
h_{x,i+1}\norm{(\eta_u)_{i,y}^{-}}_{J_j}^2
\right)^{1/2}
\left(\sum_{j=1}^{N}\sum_{i=N/2+1}^{N-1}
h_{x,i+1}^{-1}\norm{\jump{\xi_u}_{i,y}}_{J_j}^2\right)^{1/2}
\nonumber\\
&\hspace{-15mm}
\leq C\left(\sum_{j=1}^{N}
\norm{(\eta_u)_{N,y}^{-}}_{J_j}^2
\right)^{1/2}
\enorm{\bm\xi}
+
C\left(\norm{\eta_u}^2_{\Omega_{21}\cup \Omega_{22}}
+\sum_{j=1}^{N}\sum_{i=N/2+1}^{N-1}
h_{x,i+1}\norm{(\eta_u)_{i,y}^{-}}_{J_j}^2
\right)^{1/2}
\nonumber\\
&\hspace{-15mm}\qquad
\times\left(\norm{(\xi_u)_x}^2_{\Omega_{21}\cup \Omega_{22}}
+\sum_{j=1}^{N}\sum_{i=N/2+1}^{N-1}
h_{x,i+1}^{-1}\norm{\jump{\xi_u}_{i,y}}_{J_j}^2
\right)^{1/2}
\nonumber\\
&\hspace{-15mm}\leq
C\left(\sum_{j=1}^{N}
\norm{(\eta_u)_{N,y}^{-}}_{J_j}^2
\right)^{1/2}\enorm{\bm\xi}
+C\sq^{-1/2}\left(\norm{\eta_u}^2_{\Omega_{21}\cup \Omega_{22}}
+\sum_{j=1}^{N}\sum_{i=N/2+1}^{N-1}
h_{x,i+1}\norm{(\eta_u)_{i,y}^{-}}_{J_j}^2
\right)^{1/2}
\nonumber\\
\label{T43}
&\hspace{-15mm}\qquad
\times\left[\sq^{-1}\norm{\eta_p}^2+\enorm{\bm\xi}^2
+\sq\sup_{s\in \spc}
\sum_{K_{ij}\in \Omega_{21}\cup \Omega_{22}}  
\left(\frac{{\mathcal D}^1_{ij}(\eta_u;s)}{\norm{s}_{K_{ij}}}\right)^2
\right]^{1/2},
\end{align}
where we used the following inequalities
from \cite[proof of Lemma 3.4]{Cheng:Jiang:Stynes:2022}:
	\begin{align*}
	\norm{(\xi_u)_x}_{K_{ij}}
	&
	\leq 
	C\varepsilon^{-1}
	\left(\norm{\eta_p}_{K_{ij}}+\norm{\xi_p}_{K_{ij}}\right)
	+\frac{C|\mathcal{D}^{1}_{ij}(\eta_u;s_1)|}{\norm{s_1}_{K_{ij}}}\,,
	\\
	h_{x,i}^{-1}\norm{\jump{\xi_u}_{i-1,y}}^2_{J_j}
	&\leq 
	C\left[
	\varepsilon^{-2}\left(\norm{\eta_p}_{K_{ij}}^2+\norm{\xi_p}_{K_{ij}}^2\right)
	+\norm{(\xi_u)_x}_{K_{ij}}^2
	+\left(\frac{|\mathcal{D}^{1}_{ij}(\eta_u;s_2)|}{\norm{s_2}_{K_{ij}}}\right)^2
	\right]
	\end{align*}
for two particular functions $s_1,s_2\in\spc$ that are chosen to have certain properties.

Using this pair of inequalities again in a similar manner, we get
\begin{align}
\left|\mathcal{T}_{44}(\eta_{u};\xi_u)\right|
&\leq C\left(\sum_{j=1}^{N}\norm{(\eta_u)_{N/2,y}^{-}}_{J_j}^2
\right)^{1/2}
\left(\sum_{j=1}^{N}\norm{(\xi_u)_{N/2,y}^{+}}_{J_j}^2
\right)^{1/2}
\label{T44}\\
&
\hspace{-17mm}\leq 
C\sqrt{\frac{\tau_1}{\sq}}
\left(\sum_{j=1}^{N}\norm{(\eta_u)_{N/2,y}^{-}}_{J_j}^2
\right)^{1/2}
\left(\sq^{-1}\norm{\eta_p}^2+\enorm{\bm\xi}^2
+ \sq\sup_{s\in \spc}
\sum_{K_{ij}\in \Omega_{21}\cup \Omega_{22}}  
\left(\frac{{\mathcal D}^1_{ij}(\eta_u;s)}{\norm{s}_{K_{ij}}}\right)^2
\right)^{1/2},\nonumber
\end{align}
where we also used the case $i=N/2$ of the inequality
\begin{align*}
\sum_{j=1}^N\norm{(\xi_u)^+_{i,y}}^2_{J_j}
&\leq C\tau_1
\Bigg[\sum_{j=1}^N\sum_{\ell=i+1}^N \norm{(\xi_u)_x}^2_{K_{\ell j}}
+\sum_{j=1}^N\sum_{\ell=i+1}^{N-1} h^{-1}_{x,\ell+1}
\norm{\jump{\xi_u}_{\ell,y}}_{J_j}^2
\Bigg]
+3\sum_{j=1}^N\norm{\jump{\xi_u}_{N,y}}_{J_j}^2
\end{align*}
that was derived in \cite[proof of Lemma 3.5]{Cheng:Jiang:Stynes:2022}.

Finally, a Cauchy-Schwarz inequality yields
\begin{equation}\label{T45}
\left|\mathcal{T}_{45}(\eta_u;\xi_u)\right|
\le C\Bigg(\sum_{j=1}^N\norm{(\eta_u)^{-}_{N,y}}^2_{J_j}
+\sum_{i=1}^N\norm{(\eta_u)^{-}_{x,N}}^2_{I_i}\Bigg)^{1/2}
\enorm{\bm\xi}
\end{equation}
because $0\leq \lambda_{1}\leq C$ and $0< \lambda_{2}\leq C$.

%
%
%

\subsection{Uniform convergence and uniform supercloseness}
\label{sec:theorem}

We now state and prove the main result of the paper.

\begin{theorem}\label{thm:superconvergent}
Assume that $\sq^{1/4}\ln N\leq C$, that Assumption~\ref{assumption:reg:2d} is valid for $m=k$,
and that $\sigma\geq k+2$ in~\eqref{tau}.
	Let $\bm w=(u,p,q)=(u,\varepsilon u_x,\varepsilon u_y)$
	be the solution of problem \eqref{cd:spp:parabolic}.
	Let $\wN=(\uN,\pN,\qN)\in \spc^3$ be the numerical solution of the LDG method \eqref{compact:form:2d}.
	Then there exists a constant $C>0$ for which the following bounds hold true. 
	One has the energy-norm error estimate
	\begin{align}\label{energy:error}
	\enorm{\bm w-\wN}\leq  C(N^{-1}\ln N)^{k+1/2}
	\end{align}
	and the superclose property
    \begin{align}\label{super:energy:error}
    \enorm{\bm\Pi \bm w-\wN}\leq  C(N^{-1}\ln N)^{k+1},
    \end{align}
	where $\bm\Pi \bm w=(\Pi^- u,\Pi_x^+ p, \Pi_y^+ q) \in \spc^3$ 
	is the local Gauss-Radau projection of $\bm w $ defined in Section~\ref{sec:GRtrue}. 
	In particular, \eqref{super:energy:error} implies the optimal-order $L^2$ error estimate
	\begin{align}\label{optimal:L2:error}
	\lnorm{\bm w-\wN} \leq C(N^{-1}\ln N)^{k+1}.
	\end{align}
\end{theorem}
\begin{proof}
Combining \eqref{T1}--\eqref{T45}, 
we get the bound
\begin{align}
\label{A:general:estimate:B}
&|B(\bm \eta;\bm\xi)|
\leq C\Bigg\{
\varepsilon^{-1/2}\norm{\eta_p}+\varepsilon^{-1/2}\norm{\eta_q}
+\norm{\eta_u}
+\sq^{-1/2}\norm{\eta_u}_{\Omega_{21}\cup \Omega_{22}}
\nonumber\\
&\qquad
+\Bigg(\sum_{j=1}^N\norm{(\eta_u)^{-}_{N,y}}^2_{J_j}\Bigg)^{1/2}
+\Bigg(\sum_{i=1}^N\norm{(\eta_u)^{-}_{x,N}}^2_{I_i}\Bigg)^{1/2}
+\sqrt{\frac{\tau_1}{\sq}}
\Bigg(\sum_{j=1}^{N}\norm{(\eta_u)_{N/2,y}^{-}}_{J_j}^2
\Bigg)^{1/2}
\nonumber\\
&\qquad
+\sq^{-1/2}\left(
\sum_{j=1}^{N}\sum_{i=N/2+1}^{N-1}
h_{x,i+1}\norm{(\eta_u)_{i,y}^{-}}_{J_j}^2
\right)^{1/2}
+N^{-1/2}\Bigg(\sum_{j=1}^N\sum_{i=1}^{N}
\norm{(\eta_u)^{-}_{i,y}}^2_{J_{j}}
\Bigg)^{1/2}
\nonumber\\
&\hspace{-15mm}\qquad
+\left(\sum_{j=1}^N\norm{(\eta_p)^{-}_{N,y}}^2_{J_j}\right)^{1/2}
+\lambda_2^{-1/2}
\left(\sum_{i=1}^N\norm{(\eta_q)^{-}_{x,N}}^2_{I_i}\right)^{1/2}
+\left[\sum_{K_{ij}\in \Omega_{11}\cup \Omega_{12}}
\left(\frac{{\mathcal D}^1_{ij}(\eta_u;\xi_u)}{\norm{\xi_u}_{K_{ij}}}\right)^2
\right]^{1/2}
\nonumber\\
&\qquad
\left. +
\sq^{1/2}\left[ \sum_{K_{ij}\in \Omega_{N}}  
\left(\frac{{\mathcal D}^1_{ij}(\eta_u;\xi_p)}{\norm{\xi_p}_{K_{ij}}}\right)^2
+\sum_{K_{ij}\in \Omega_{N}}  
\left(\frac{{\mathcal D}^2_{ij}(\eta_u;\xi_q)}{\norm{\xi_q}_{K_{ij}}}\right)^2
\right] ^{1/2}
\right\}
\nonumber\\
%
&\qquad\qquad
\times\left\{
\sq^{-1/2}\norm{\eta_p}+\enorm{\bm\xi}
+\sq^{1/2}
\left[\sup_{s\in \spc}
\sum_{K_{ij}\in \Omega_{21}\cup \Omega_{22}}  
\left(\frac{{\mathcal D}^1_{ij}(\eta_u;s)}{\norm{s}_{K_{ij}}}\right)^2
\right]^{1/2}
\right\}.
\end{align}
Invoking \eqref{etau:bry:left:x}, one has
\begin{align*}
\sq^{-1/2}\left(
\sum_{j=1}^{N}\sum_{i=N/2+1}^{N-1}
h_{x,i+1}\norm{(\eta_u)_{i,y}^{-}}_{J_j}^2
\right)^{1/2}
&\leq \sqrt{\frac{\tau_1}{\sq}}
\max_{1\leq i\leq N}
\Bigg(\sum_{j=1}^{N}\norm{(\eta_u)_{i,y}^{-}}_{J_j}^2
\Bigg)^{1/2}
\\
&\leq C(\ln N)^{1/2}
\left[N^{-(k+1)}+\sq^{1/4}(N^{-1}\ln N)^{k+1}\right]
\\
&
\leq C(N^{-1}\ln N)^{k+1},
\end{align*}
since $\sq^{1/4}(\ln N)^{1/2} \le \sq^{1/4}\ln N\leq C$ by hypothesis. 
Substituting this inequality,   
$\lambda_2\geq C_1\varepsilon$, $\sigma\geq k+2$
and the results of Lemmas~\ref{lemma:GR1},
\ref{lemma:GR2}
and~\ref{superapproximation:element}
into \eqref{A:general:estimate:B},
we obtain 
\[
\left|B(\bm \eta;\bm\xi)\right|
\leq C(N^{-1}\ln N)^{k+1}\left[\enorm{\bm\xi}
	+(N^{-1}\ln N)^{k+1}\right].
\]
Thus from \eqref{error:equation} it follows that
$\enorm{\bm\xi}^2\leq C(N^{-1}\ln N)^{2(k+1)}+\frac12\enorm{\bm\xi}^2$,
whence
\begin{align}
\label{super:xi:1}\enorm{\bm\xi}&\leq C(N^{-1}\ln N)^{k+1},
\end{align}
which proves \eqref{super:energy:error}.

The energy-norm error estimate \eqref{energy:error}
and $L^2$ error estimate \eqref{optimal:L2:error} 
now follow immediately from \eqref{super:xi:1}  and 
Lemmas~\ref{lemma:GR1} and \ref{lemma:GR2} 
since $ \bm w - \bm W = \bm \eta-\bm \xi$.
\end{proof}

\begin{remark}[Extension to other layer-adapted meshes]\label{Remark:B:BS:mesh}
We can extend our analysis 
to other types of layer-adapted meshes
such as those of  Bakhvalov-Shishkin type (BS-mesh)
and Bakhvalov type (B-mesh).
To achieve this, one needs properties analogous to those 
proved in Lemmas~\ref{lemma:GR1},
\ref{lemma:GR2} and~\ref{superapproximation:element},
which can be derived if one has 
\[
\max_{1\leq i,j \leq N}\left\{h_{x,i},h_{y,j}\right\}\leq CN^{-1}
\ \text{ and } \ 
	\max\left\{\left[\psi_1\left(\frac12\right)\right]^{\sigma},
	\left[\psi_2\left(\frac14\right)\right]^{\sigma}
	\right\}\leq CN^{-(k+2),
}
\]
where the $\psi_i$ ($i=1,2$)
are the mesh-characterizing functions (see \cite[Section 2.1]{Linss10})
in each coordinate direction.
These two inequalities can be ensured 
for the B-mesh by choosing a sufficiently large
$\sigma$, but for the BS-mesh 
our analysis needs to assume 
$\ssq\leq N^{-1}$ to derive the first inequality. 
The numerical results in Section~\ref{subsec:assumption}
show that the condition $\ssq\leq N^{-1}$ is not artificial; 
only when it is satisfied does
the BS-mesh attain the optimal convergence rates
for $\lnorm{\bm w-\wN}$ and $\enorm{\bm\Pi \bm w-\wN}$.

See \cite[Remark 4.4]{Cheng:Jiang:Stynes:2022} for more details
of this extension.
\end{remark}

\section{Numerical experiments}
\label{sec:experiments}
\setcounter{equation}{0}

In this section,
we present some numerical results
for the LDG method applied to the following problem:
\begin{align*}			
-\varepsilon \Delta u + (1+x)(1+y) u_x + \left(\frac32+y\right)u&= f 
\ \text{ in } \Omega=(0,1)^2,
\\
u &= 0\ \text{ on } \partial \Omega,
\end{align*}
where the right-hand side $f$ is chosen such that
\[
u(x,y)
=\left(\sin\frac{\pi x}{2}
-\frac{e^{-(1-x)/\varepsilon}-e^{-1/\varepsilon}}{1-e^{-1/\varepsilon}}\right)
\left(1+y^4\right)
\frac{\left(1-e^{-y/\ssq}\right)\left(1-e^{-(1-y)/\ssq}\right)}{\left(1-e^{-1/(2\ssq)}\right)^2}
\]
is the true solution. Obviously this solution satisfies Assumption \ref{assumption:reg:2d} for any non-negative integer $m$.

In all our computations 
we take $\sigma = k+2$, $\alpha=1$ and $\delta=1.4$ 
in \eqref{tau:used} and choose 
the penalty parameters $\lambda_1=0$ and $\lambda_{2}=\sq$ in the LDG method.
The discrete linear systems are solved using LU decomposition, i.e., 
a direct linear solver.
All integrals are evaluated using the 5-point Gauss-Legendre quadrature rule.

We compute the three errors
\begin{align}\label{three:errors}
\lnorm{\bm{w}-\wN},\quad \enorm{\Pi \bm w-\wN}
\text{ and } \quad \enorm{\bm w-\wN}.
\end{align}
Below, $E_N$ is used to denote each of these errors when $N$ elements are used in each coordinate direction.

On the Shishkin mesh our theoretical bounds are of the form ${\interleave \text{error} \interleave} \le C(N^{-1}\ln N)^{r_S}$,
so we estimate $r_S$ by computing the numerical convergence rate
\[
r_{S,N} := \frac{ \log (E_N/E_{2N}) }{\log \big(2\ln N/\ln(2N)\big) }
\]
for various values of $N$.
On the other types of layer-adapted meshes (see Section~\ref{subsec:BS:B:mesh}) where one would expect 
${\interleave \text{error} \interleave}\le CN^{-r_2}$, we estimate $r_2$ in the standard way by computing the 
numerical convergence rate $r_{2,N}:=\log (E_N/E_{2N}) /\log 2$.

In Section \ref{subsec:S:mesh}, we compute all three errors of \eqref{three:errors} on Shishkin meshes
to confirm the sharpness of Theorem \ref{thm:superconvergent}.
In Section \ref{subsec:BS:B:mesh}, we test the uniform convergence and 
supercloseness of the LDG method on the BS-mesh and B-mesh.
In  Section \ref{subsec:assumption}, we investigate the effect of the condition $\ssq\leq N^{-1}$ 
on the convergence rates of the LDG method on the BS-mesh (recall Remark~\ref{Remark:B:BS:mesh}).


\subsection{Uniform convergence and supercloseness}
\label{subsec:S:mesh}

For $k=2$ and  $N=64$, 
Figure \ref{figure:1} displays the numerical solution 
and the pointwise error  $U-u$ computed by the LDG method on a Shishkin mesh 
for $\varepsilon=10^{-4}$ and $\varepsilon=10^{-8}$.
One sees that the LDG method produces quite good results --- 
the layers are sharp and no obvious oscillations appear anywhere in the domain.

\begin{figure}[h]
	\begin{minipage}{0.49\linewidth}
		\centerline{\includegraphics[width=2.6in,height=2.4in]{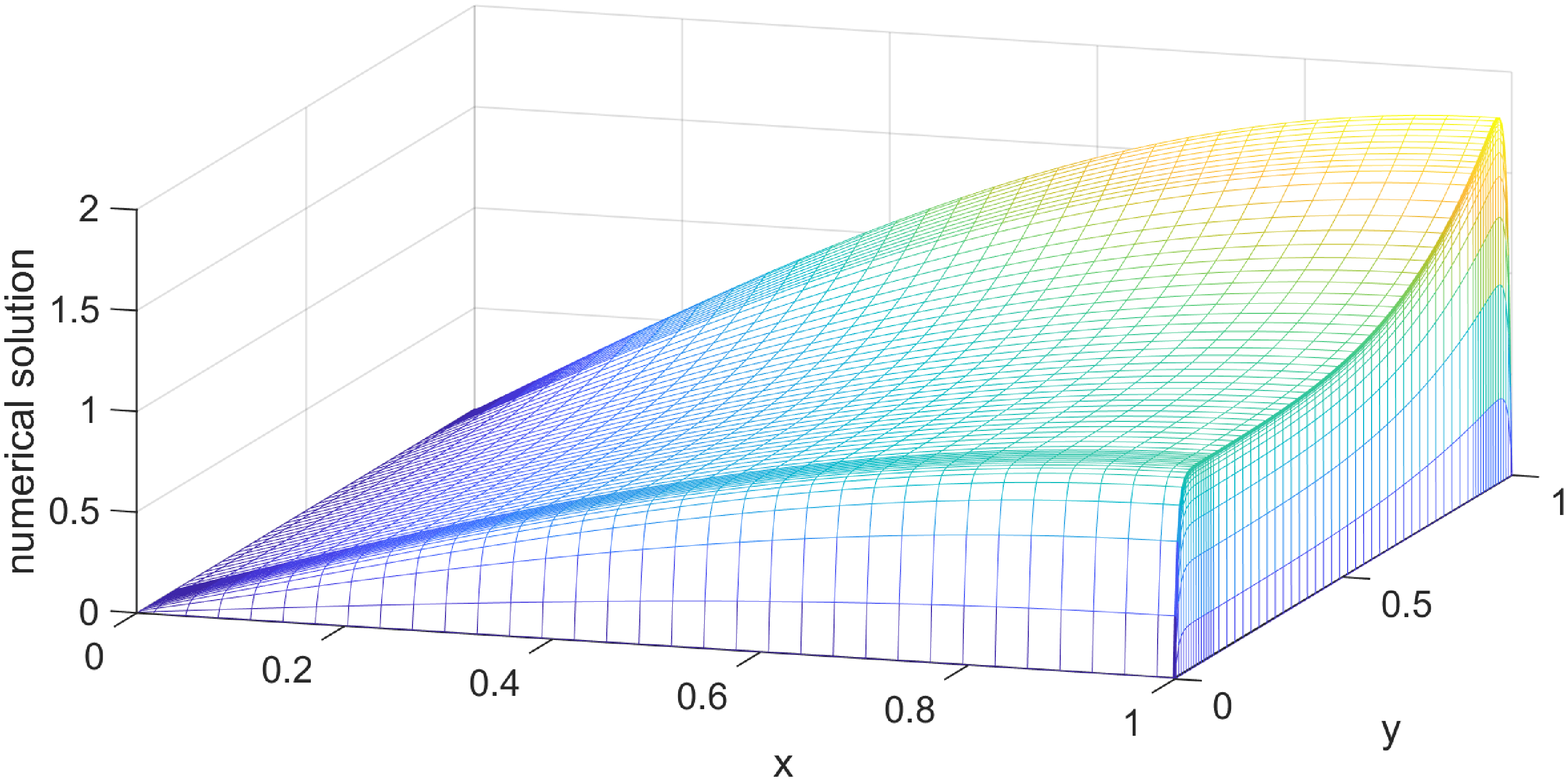}}
		\centerline{(a)}
		\centerline{\includegraphics[width=2.6in,height=2.4in]{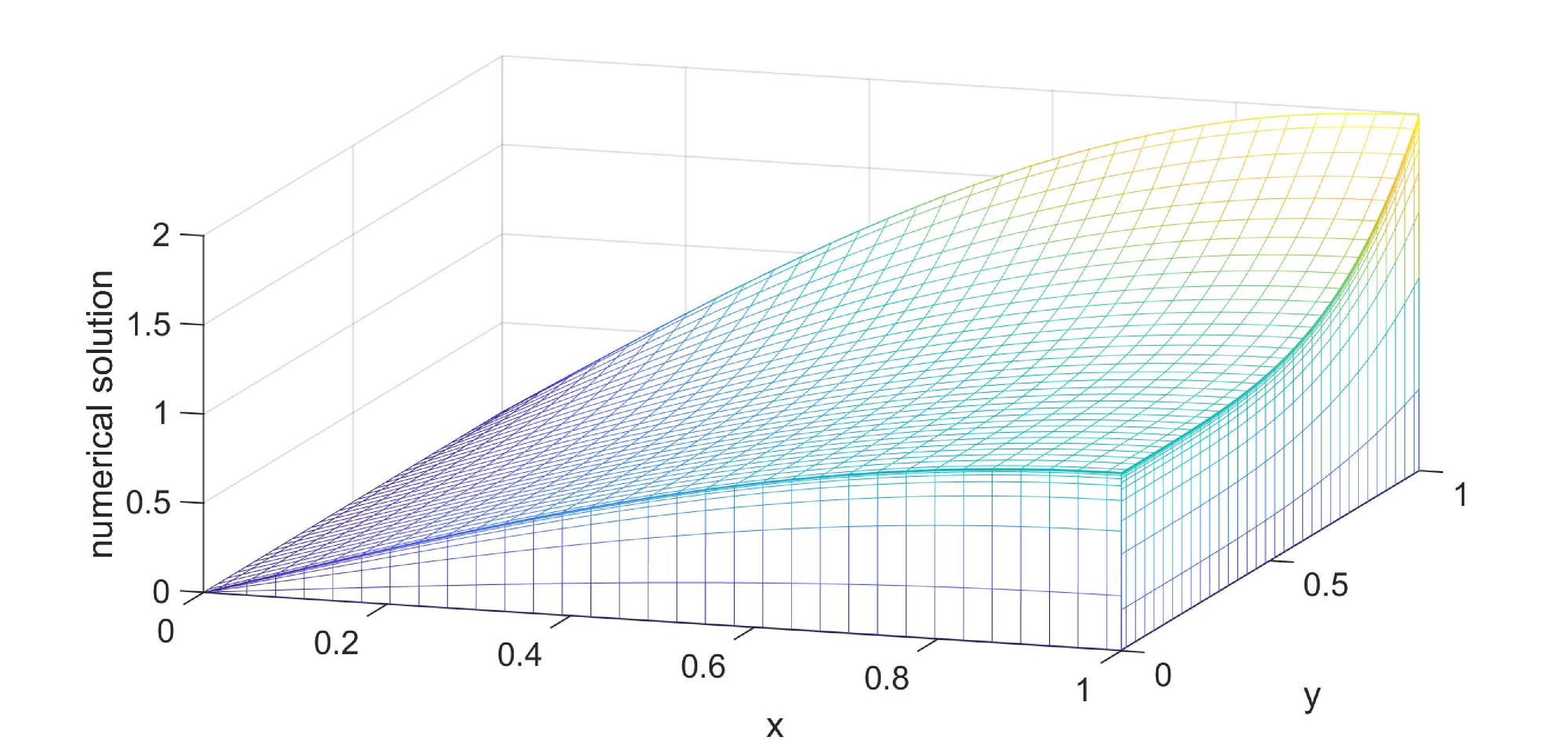}}
		\centerline{(b)}
	\end{minipage}
	\begin{minipage}{0.49\linewidth}
		\centerline{\includegraphics[width=2.6in,height=2.4in]{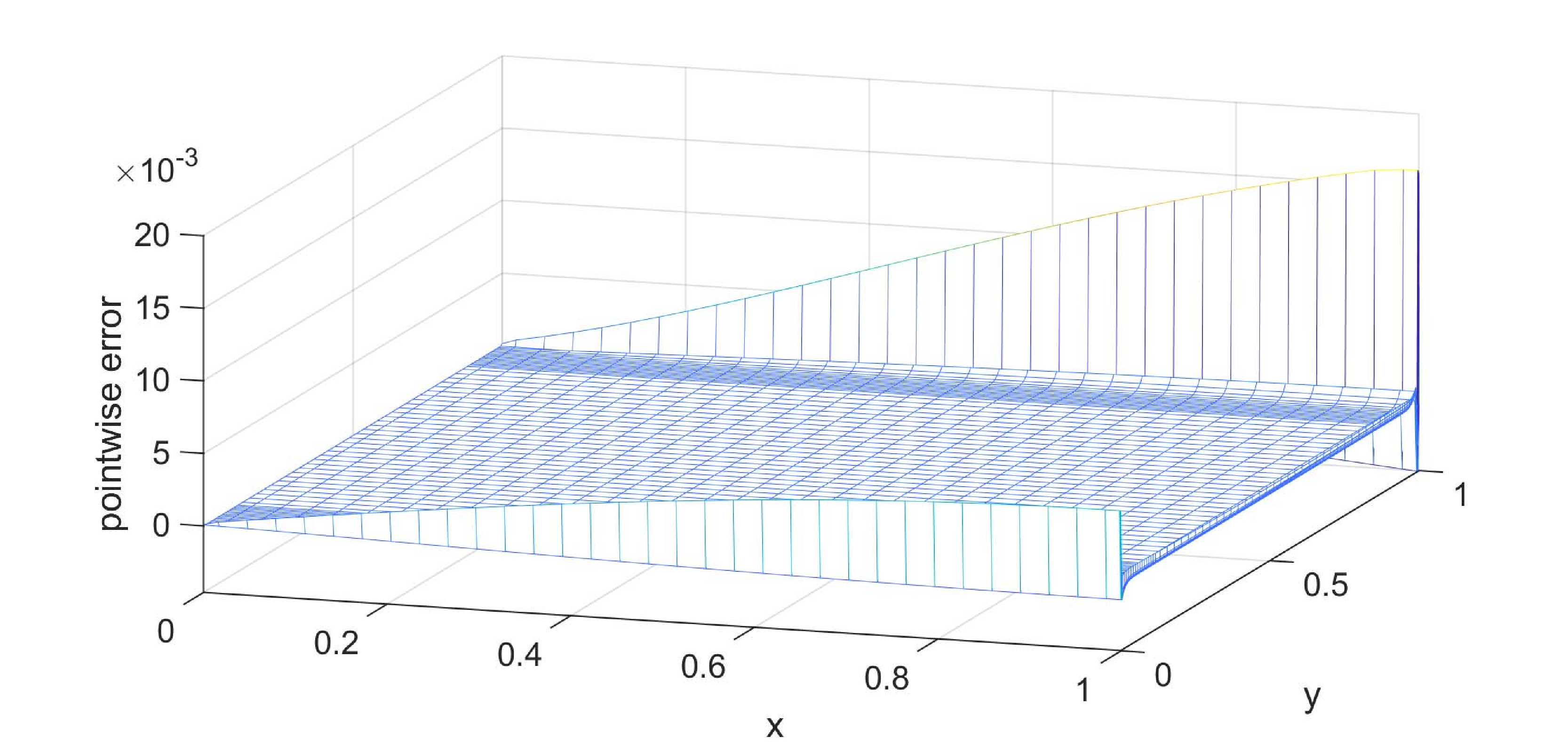}}
		\centerline{(c)}
		\centerline{\includegraphics[width=2.6in,height=2.4in]{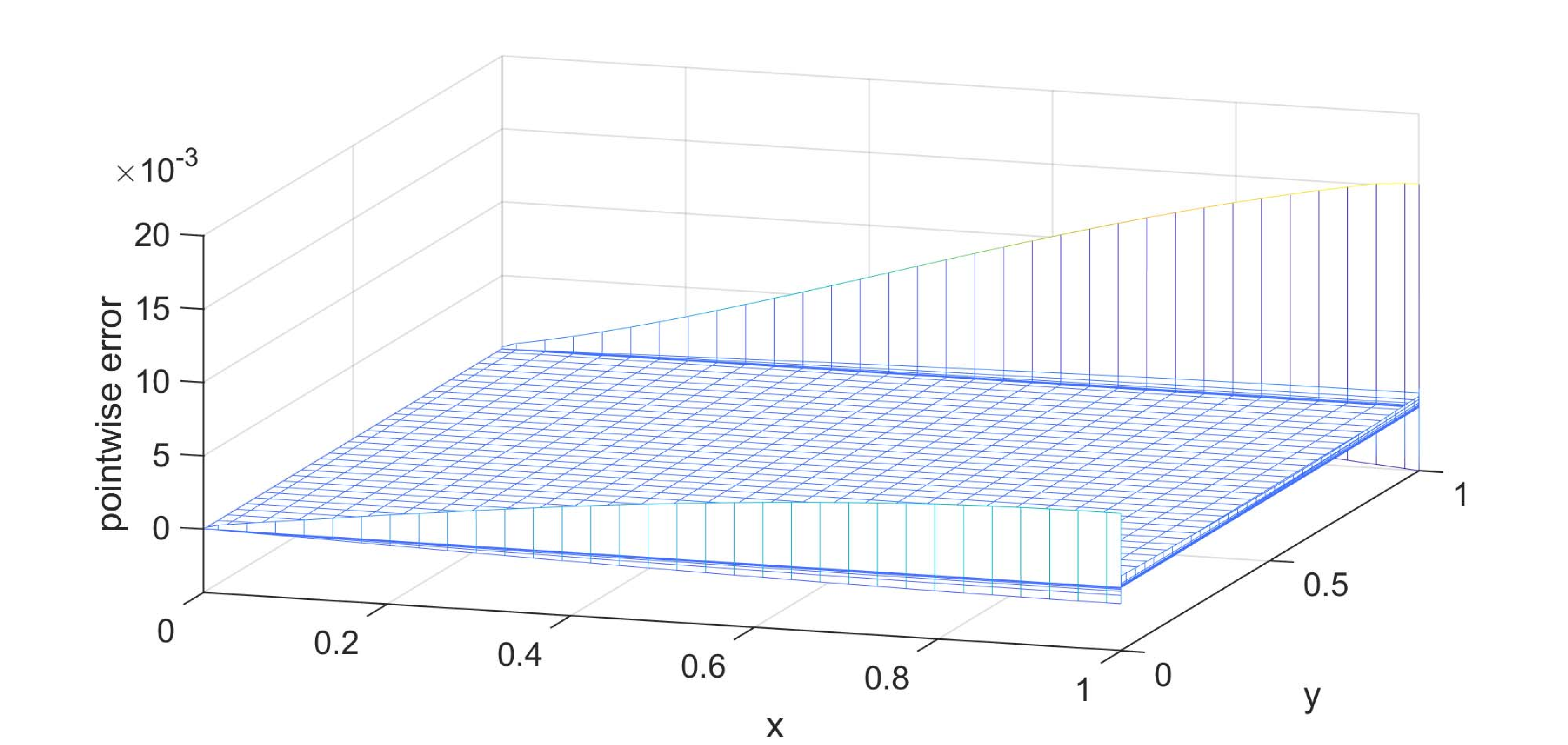}}
		\centerline{(d)}
	\end{minipage}
	\caption{
		(a) numerical solution $U$ for $\varepsilon=10^{-4}$,
		(b) numerical solution $U$ for $\varepsilon=10^{-8}$,
		(c) pointwise error $U-u$ for $\varepsilon=10^{-4}$,
		(d) pointwise error $U-u$ for $\varepsilon=10^{-8}$.
	}
	\label{figure:1}
\end{figure}

For $\varepsilon=10^{-8}$, in
Table \ref{table:S-mesh:rate} we list the three errors of \eqref{three:errors} for $k=0,1,2,3$
and various values of~$N$.
One observes that the energy-norm error
$\enorm{\bm w-\wN}$ in the true solution converges at a rate of
$O((N^{-1}\ln N)^{k+1/2})$, while the $L^2$ error $\lnorm{\bm{w}-\wN}$
and the energy-norm error $\enorm{\Pi \bm w-\wN}$ in the Gauss-Radau projection of the true solution
both converge at a rate of $O((N^{-1}\ln N)^{k+1})$.
These results agree exactly with the rates predicted 
in Theorem~\ref{thm:superconvergent}.
These rates include the piecewise-constant case $k=0$,
although our theoretical analysis doesn't cover this variant of the LDG.

To test the the robustness of these errors 
with respect to the singular perturbation parameter $\varepsilon$, 
we fix $k=2$ and  $N=128$
then test $\sq = 10^{-4},10^{-5},\dots,10^{-10}$.
Table~\ref{table:S-mesh:robust} 
shows that all three LDG solution errors of \eqref{three:errors}
are robust as $\sq\rightarrow 0$, which agrees with our theory.

\subsection{Two other layer-adapted meshes}
\label{subsec:BS:B:mesh}

In this subsection we test the numerical convergence rates of the LDG method on
the other layer-adapted meshes mentioned in Remark~\ref{Remark:B:BS:mesh},
viz., the BS-mesh and B-mesh.
Now the mesh points $(x_i,y_j)$ for $i,j=0,1,\dots, N$ are defined by
\begin{align*}
x_i&:=
\begin{dcases}
2(1-\tau_1)\frac{i}{N}   &\qquad\text{for } i=0,1,...,N/2,
\\
1-\frac{\sigma\varepsilon}{\alpha}\varphi_1\left(1-\frac{i}{N}\right)
&\qquad\text{for } i=N/2+1,N/2+2,...,N,
\end{dcases}
\\
y_j&:=
\begin{dcases}
\frac{\sigma\ssq}{\delta}\varphi_2\left(\frac{j}{N}\right)   &\text{for } j=0,1,...,N/4,
\\
\tau_2 + 2(1-2\tau_2)\left(\frac{j}{N}-\frac14\right)
&\text{for } j=N/4+1,N/4+2,...,3N/4,
\\
1-\frac{\sigma\ssq}{\delta}\varphi_2\left(1-\frac{j}{N}\right)
&\text{for } j=3N/4+1,3N/4+2,...,N,
\end{dcases}
\end{align*}
where the mesh-generating function $(\varphi_1,\varphi_2)$ is given by
\[
\big(\varphi_1(t),\varphi_2(t)\big)	
=\begin{cases}
\big(-\ln(1-2(1-N^{-1})t),-\ln(1-4(1-N^{-1})t)\big)
&\text{for the BS-mesh,}\\
\big(-\ln(1-2(1-\sq)t),-\ln(1-4(1-\ssq)t)\big)    
&\text{for the B-mesh.}
\end{cases}
\]

Tables \ref{table:B-mesh:rate} and \ref{table:BS-mesh:rate} 
show that for both these meshes
the three errors of \eqref{three:errors} 
converge at rates of $O(N^{-(k+1)})$, 
$O(N^{-(k+1)})$ and $O(N^{-(k+1/2)})$ respectively.
Furthermore, the robustness of these errors with respect to~$\varepsilon$ 
is demonstrated by Tables \ref{table:B-mesh:robust} and \ref{table:BS-mesh:robust}
where $k=2$ and $N=128$.

\begin{table}[ht]
	\footnotesize
	\centering
	\caption{Shishkin mesh}
	\smallskip
	\label{table:S-mesh:rate}
	\begin{tabular}{cccccccc}
		\toprule
		& $N$ & $\lnorm{\bm{w}-\wN}$ & rate & $\enorm{\Pi \bm{w}-\wN}$ & rate  
		& $\enorm{\bm{w}-\wN}$ & rate 
		\\
		\midrule
		$\mathcal{P}^0$
		& 4   & 5.3536e-01 & --      & 7.9351e-01 & --      & 1.3312e+00 & --      \\
		& 8   & 3.5217e-01 & 1.4558 & 5.9786e-01 & 0.9841 & 1.1198e+00 & 0.6010 \\
		& 16  & 2.2387e-01 & 1.1174 & 4.1177e-01 & 0.9197 & 9.0791e-01 & 0.5174 \\
		& 32  & 1.3930e-01 & 1.0093 & 2.7262e-01 & 0.8774 & 7.2416e-01 & 0.4811 \\
		& 64  & 8.4617e-02 & 0.9759 & 1.7412e-01 & 0.8777 & 5.6927e-01 & 0.4711 \\
		& 128 & 5.0108e-02 & 0.9721 & 1.0728e-01 & 0.8985  & 4.4091e-01 & 0.4740 \\
		& 256 & 2.8977e-02 & 0.9787 & 6.3918e-02 & 0.9254 & 3.3674e-01 & 0.4817 \\
		$\mathcal{P}^1$
		& 4   & 1.3850e-01 & --      & 2.0298e-01 & --      & 3.8324e-01 & --      \\
		& 8   & 8.0738e-02 & 1.8760 & 1.3090e-01 & 1.5249 & 2.5007e-01 & 1.4840 \\
		& 16  & 4.0740e-02 & 1.6870 & 7.0952e-02 & 1.5105 & 1.4509e-01 & 1.3426 \\
		& 32  & 1.7817e-02 & 1.7596 & 3.3128e-02 & 1.6205 & 7.6860e-02 & 1.3518 \\
		& 64  & 6.9845e-03 & 1.8333 & 1.3712e-02 & 1.7268 & 3.8200e-02 & 1.3687 \\
		& 128 & 2.5214e-03 & 1.8903 & 5.1503e-03 & 1.8168 & 1.7999e-02 & 1.3962 \\
		& 256 & 8.5645e-04 & 1.9295 & 1.7939e-03 & 1.8846 & 8.0968e-03 & 1.4275 \\
		$\mathcal{P}^2$
		& 4   & 4.3661e-02 & --      & 6.8852e-02 & --      & 1.2545e-01 & --      \\
		& 8   & 2.2194e-02 & 2.3520 & 3.6989e-02 & 2.1598 & 6.9999e-02 & 2.0280 \\
		& 16  & 7.9497e-03 & 2.5321 & 1.4701e-02 & 2.2756 & 2.8911e-02 & 2.1809 \\
		& 32  & 2.2294e-03 & 2.7051 & 4.6012e-03 & 2.4715 & 9.8296e-03 & 2.2954 \\
		& 64  & 5.2857e-04 & 2.8177 & 1.1956e-03 & 2.6383 & 2.9669e-03 & 2.3450 \\
		& 128 & 1.1159e-04 & 2.8857 & 2.6906e-04 & 2.7671 & 8.2314e-04 & 2.3788 \\
		& 256 & 2.1761e-05 & 2.9211 & 5.4376e-05 & 2.8573 & 2.1323e-04 & 2.4137 \\
		$\mathcal{P}^3$
		& 4   & 1.4625e-02 & --      & 2.3754e-02 & --      & 4.3459e-02 & --      \\
		& 8   & 6.1231e-03 & 3.0263 & 1.0418e-02 & 2.8649 & 1.9649e-02 & 2.7593 \\
		& 16  & 1.5761e-03 & 3.3471 & 3.0020e-03 & 3.0688 & 5.7651e-03 & 3.0242 \\
		& 32  & 2.8816e-04 & 3.6153  & 6.2676e-04 & 3.3329 & 1.2680e-03 & 3.2220 \\
		& 64  & 4.1807e-05 & 3.7791 & 1.0209e-04 & 3.5525 & 2.3327e-04 & 3.3143 \\
		& 128 & 5.2232e-06 & 3.8589 & 1.3776e-05 & 3.7160 & 3.8117e-05 & 3.3609 \\
		& 256 & 5.9333e-07 & 3.8868 & 1.6310e-06 & 3.8129 & 5.6677e-06 &3.4057       \\
		\bottomrule
	\end{tabular}
\end{table}

\begin{table}[ht]
	\small
	\centering
	\caption{Robustness: Shishkin mesh.}
	\smallskip
	\label{table:S-mesh:robust}
	\begin{tabular}{cccc}
		\toprule
		$\varepsilon$ & $\lnorm{\bm{w}-\wN}$  & $\enorm{\Pi \bm{w}-\wN}$   
		& $\enorm{\bm{w}-\wN}$  
		\\
		\midrule
		$10^{-4}$  & 2.3131e-04 & 3.2951e-04 & 8.2653e-04 \\
		$10^{-5}$  & 1.5557e-04 & 2.8749e-04 & 8.2341e-04 \\
		$10^{-6}$  & 1.2578e-04 & 2.7446e-04 & 8.2315e-04 \\
		$10^{-7}$  & 1.1513e-04 & 2.7036e-04 & 8.2313e-04 \\
		$10^{-8}$  & 1.1159e-04 & 2.6906e-04 & 8.2314e-04 \\
		$10^{-9}$  & 1.1044e-04 & 2.6874e-04 & 8.2317e-04 \\
		$10^{-10}$ & 1.1007e-04 & 2.6982e-04 & 8.2100e-04 \\
		\bottomrule
	\end{tabular}
\end{table}

\begin{table}[ht]
	\footnotesize
	\centering
	\caption{Bakhvalov-type mesh.}
	\smallskip
	\label{table:B-mesh:rate}
	\begin{tabular}{cccccccc}
		\toprule
		& $N$ & $\lnorm{\bm{w}-\wN}$ & rate & $\enorm{\Pi \bm{w}-\wN}$ & rate  
		& $\enorm{\bm{w}-\wN}$ & rate 
		\\
		\midrule
		$\mathcal{P}^1$
		& 4   & 1.4205e-01 & --      & 2.6525e-01 & --      & 3.9844e-01 & --      \\
		& 8   & 4.1333e-02 & 1.7810  & 6.0664e-02 & 2.1284 & 1.4973e-01 & 1.4120 \\
		& 16  & 1.1469e-02 & 1.8496 & 1.5490e-02 & 1.9695 & 5.6414e-02 & 1.4082 \\
		& 32  & 3.0832e-03 & 1.8953 & 3.9969e-03 & 1.9544 & 2.0875e-02 & 1.4343 \\
		& 64  & 8.0743e-04 & 1.9330 & 1.0201e-03 & 1.9702 & 7.5866e-03 & 1.4603 \\
		& 128 & 2.0782e-04 & 1.9580 & 2.5811e-04 & 1.9826 & 2.7233e-03 & 1.4781 \\
		& 256 & 5.3026e-05 & 1.9705 & 6.5095e-05 & 1.9874 & 9.7055e-04 & 1.4885 \\
		$\mathcal{P}^2$
		& 4   & 6.0231e-02 & --      & 1.0894e-01 & --      & 1.4710e-01 & --      \\
		& 8   & 7.4372e-03 & 3.0177 & 1.1629e-02 & 3.2278 & 2.5974e-02 & 2.5017 \\
		& 16  & 9.8442e-04 & 2.9174 & 1.4508e-03 & 3.0027 & 4.8192e-03 & 2.4302 \\
		& 32  & 1.3178e-04 & 2.9012 & 1.8653e-04 & 2.9594 & 8.8924e-04 & 2.4382 \\
		& 64  & 1.7377e-05 & 2.9228 & 2.3824e-05 & 2.9689 & 1.6148e-04 & 2.4612 \\
		& 128 & 2.2664e-06 & 2.9387 & 3.0268e-06 & 2.9766 & 2.8978e-05 & 2.4784 \\
		& 256 & 2.9464e-07 & 2.9434 & 3.8870e-07 & 2.9611  & 5.1793e-06 & 2.4841 \\
		\bottomrule
	\end{tabular}
\end{table}

\begin{table}[ht]
	\footnotesize
	\centering
	\caption{Bakhvalov-Shishkin mesh.}
	\smallskip
	\label{table:BS-mesh:rate}
	\begin{tabular}{cccccccc}
		\toprule
		& $N$ & $\lnorm{\bm{w}-\wN}$ & rate & $\enorm{\Pi \bm{w}-\wN}$ & rate  
		& $\enorm{\bm{w}-\wN}$ & rate 
		\\
		\midrule
		$\mathcal{P}^1$
		& 4   & 1.0089e-01 & --      & 1.2825e-01 & --      & 2.8827e-01 & --      \\
		& 8   & 3.4055e-02 & 1.5669 & 4.4547e-02 & 1.5256 & 1.2773e-01 & 1.1744 \\
		& 16  & 1.0344e-02 & 1.7191 & 1.3467e-02 & 1.7259 & 5.2037e-02 & 1.2954 \\
		& 32  & 2.9197e-03 & 1.8249 & 3.7360e-03 & 1.8499 & 2.0032e-02 & 1.3772 \\
		& 64  & 7.8481e-04 & 1.8954 & 9.8633e-04 & 1.9214 & 7.4294e-03 & 1.4310 \\
		& 128 & 2.0479e-04 & 1.9382 & 2.5378e-04 & 1.9585 & 2.6946e-03 & 1.4631 \\
		& 256 & 5.2629e-05 & 1.9602 & 6.4539e-05 & 1.9753 & 9.6539e-04 & 1.4809 \\
		$\mathcal{P}^2$
		& 4   & 2.1897e-02 & --      & 3.2863e-02 & --      & 6.9537e-02 & --      \\
		& 8   & 4.6249e-03 & 2.2432 & 6.8197e-03 & 2.2687 & 1.8666e-02 & 1.8974 \\
		& 16  & 7.7361e-04 & 2.5797 & 1.1144e-03 & 2.6134 & 4.1143e-03 & 2.1817 \\
		& 32  & 1.1449e-04 & 2.7563 & 1.6005e-04 & 2.7997 & 8.2215e-04 & 2.3232  \\
		& 64  & 1.5820e-05 & 2.8555 & 2.1531e-05 & 2.8940 & 1.5526e-04 & 2.4047 \\
		& 128 & 2.1107e-06 & 2.9059 & 2.8010e-06 & 2.9424 & 2.8410e-05 & 2.4502 \\
		& 256 & 2.7838e-07 & 2.9226 & 3.7039e-07 & 2.9188 & 5.1142e-06 & 2.4738 \\
		\bottomrule
	\end{tabular}
\end{table}

\begin{table}[ht]
	\small
	\centering
	\caption{Robustness: Bakhvalov-type mesh.}
	\smallskip
	\label{table:B-mesh:robust}
	\begin{tabular}{cccc}
		\toprule
		$\varepsilon$ & $\lnorm{\bm{w}-\wN}$  & $\enorm{\Pi \bm{w}-\wN}$   
		& $\enorm{\bm{w}-\wN}$  
		\\
		\midrule
		$10^{-4}$  & 6.6843e-06 & 5.8786e-06 & 2.8893e-05 \\
		$10^{-5}$  & 4.8148e-06 & 4.5978e-06 & 2.9044e-05 \\
		$10^{-6}$  & 3.4535e-06 & 3.7717e-06 & 2.9019e-05 \\
		$10^{-7}$  & 2.6507e-06 & 3.2808e-06 & 2.8891e-05 \\
		$10^{-8}$  & 2.2664e-06 & 3.0268e-06 & 2.8978e-05 \\
		$10^{-9}$  & 2.1128e-06 & 2.9056e-06 & 2.8976e-05 \\
		$10^{-10}$ & 2.1110e-06 & 4.4365e-06 & 2.8483e-05 \\
		\bottomrule
	\end{tabular}
\end{table}

\begin{table}[ht]
	\small
	\centering
	\caption{Robustness: Bakhvalov-Shishkin mesh.}
	\smallskip
	\label{table:BS-mesh:robust}
	\begin{tabular}{cccc}
		\toprule
		$\varepsilon$ & $\lnorm{\bm{w}-\wN}$  & $\enorm{\Pi \bm{w}-\wN}$   
		& $\enorm{\bm{w}-\wN}$  
		\\
		\midrule
		$10^{-4}$  & 6.7112e-06 & 5.8800e-06 & 2.8381e-05 \\
		$10^{-5}$  & 4.5022e-06 & 4.2654e-06 & 2.8451e-05 \\
		$10^{-6}$  & 3.0409e-06 & 3.3159e-06 & 2.8423e-05 \\
		$10^{-7}$  & 2.3689e-06 & 2.9373e-06 & 2.8411e-05 \\
		$10^{-8}$  & 2.1107e-06 & 2.8010e-06 & 2.8410e-05 \\
		$10^{-9}$  & 2.0249e-06 & 2.8245e-06 & 2.8535e-05 \\
		$10^{-10}$ & 2.0229e-06 & 4.7939e-06 & 2.8153e-05 \\
		\bottomrule
	\end{tabular}
\end{table}

\subsection{Assumption $\ssq\leq N^{-1}$ on the convergence rates}
\label{subsec:assumption}

In Remark~\ref{Remark:B:BS:mesh} an additional assumption  $\ssq\leq N^{-1}$ was needed
for the BS-mesh. 
To investigate whether this condition affects the numerical convergence rates,
we fix $k=1$ (so $\sigma=3$) and test $N=60, 80,...,220$.
The numerical convergence rates are computed by
$r_{2,N_1}:=\log (E_{N_1}/E_{N_2}) /\log (N_2/N_1)$.
We choose $\ssq=0.02$, $0.01$, $0.0025$ so that our assumption  in~\eqref{tau:used}  
that $\ssq\le \delta(4\sigma\ln N)^{-1}=1.4(12\ln N)^{-1}$ is always satisfied,
but the three different regimes   
$\ssq> N^{-1}, \ \ssq\approx N^{-1}$ and $\ssq< N^{-1}$ are tested;
see Figure \ref{fig:relation}.

\begin{figure}[h]
	\begin{minipage}{0.995\linewidth}
		\centerline{\includegraphics[width=1\textwidth]{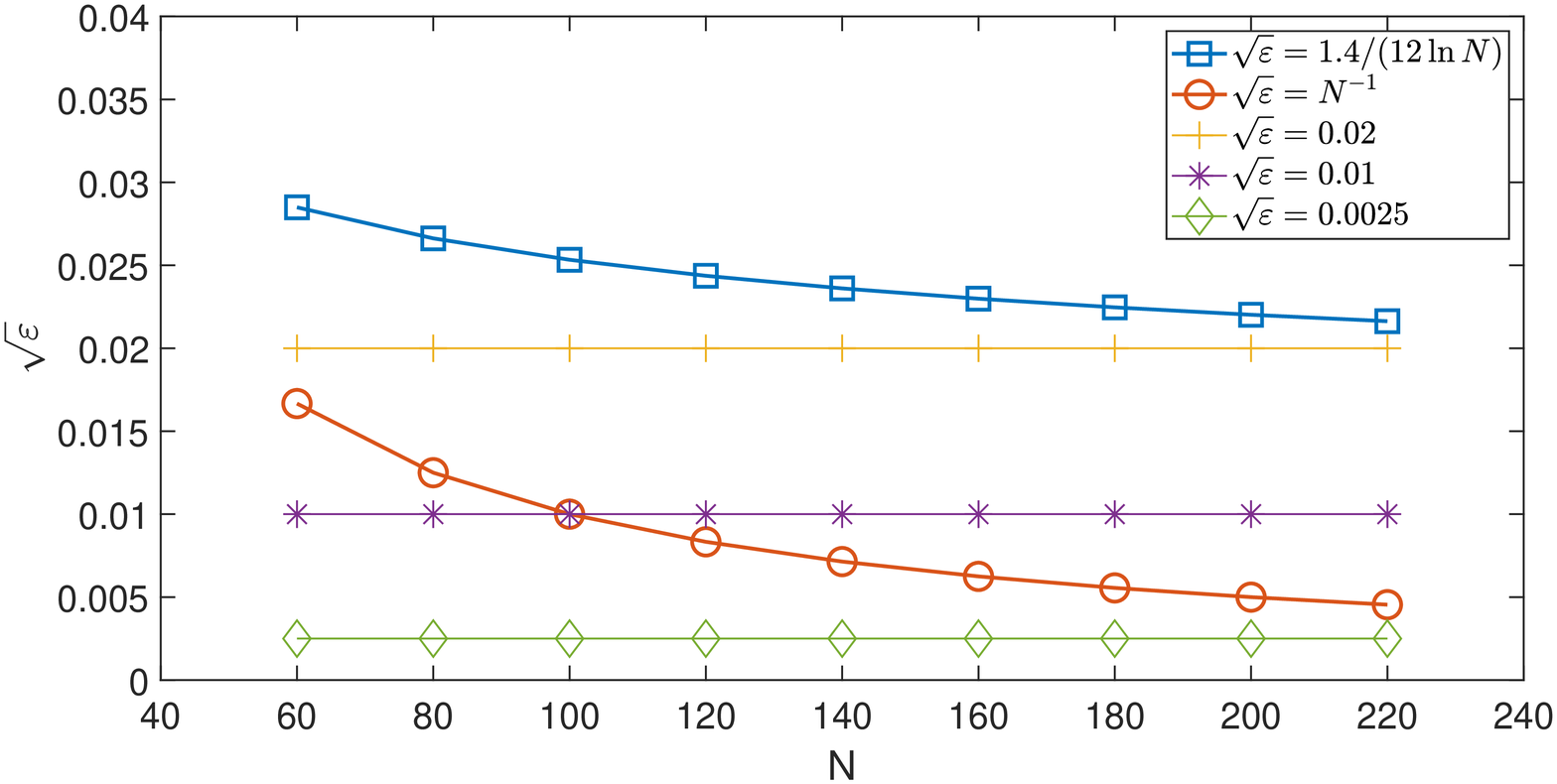}}
	\end{minipage}
	\caption{Three cases $\ssq=0.0025,0.01,0.02$.}
	\label{fig:relation}
\end{figure}

The numerical results on the BS-mesh in Table~\ref{table:BS-mesh:relation} are quite revealing.
When $\ssq = 0.02 > N^{-1}$,
there is almost no convergence for the first two errors of~\eqref{three:errors}, 
but when $\ssq = 0.0025< N^{-1}$,
optimal convergence rates are clearly seen for all three errors in~\eqref{three:errors}.
In the intermediate case $\ssq = 0.01$,
we see that the convergence rate of the first two errors of \eqref{three:errors} is optimal 
when $\ssq = 0.01\leq N^{-1}$ but then decreases sharply when this condition is violated.
The energy-norm error  $\enorm{\bm w-\wN}$ seems to be less sensitive to the relative sizes of $\varepsilon$ and $N^{-1}$,
perhaps because the boundary jump of the error dominates the whole
energy-error.
Our conclusion is that the condition $\ssq\leq N^{-1}$ is 
necessary both theoretically and in practice to obtain optimal convergence rates on the BS-mesh.

On the Shishkin and B-meshes, further numerical experiments (not presented here) show that the numerical convergence rates  are unaffected by whether or not  $\ssq\leq N^{-1}$.

\begin{table}[ht]
	\footnotesize
	\centering
	\caption{Bakhvalov-Shishkin mesh.}
	\smallskip
	\label{table:BS-mesh:relation}
	\begin{tabular}{cccccccc}
		\toprule
		& $N$ & $\lnorm{\bm{w}-\wN}$ & rate & $\enorm{\Pi \bm{w}-\wN}$ & rate  
		& $\enorm{\bm{w}-\wN}$ & rate 
		\\
		\midrule
		$\ssq=0.02$
		& 60   & 1.7777e-03 & --      & 1.7139e-03 & --      & 7.8797e-03 & --      \\
		& 80   & 1.1408e-03 & 1.5421 & 1.0893e-03 & 1.5756 & 5.1875e-03 & 1.4531 \\
		& 100  & 8.7516e-04 & 1.1878 & 8.2101e-04 & 1.2670 & 3.7573e-03 & 1.4455 \\
		& 120  & 7.5271e-04 & 0.8267 & 6.9237e-04 & 0.9347 & 2.8962e-03 & 1.4277 \\
		& 140 & 6.9128e-04 & 0.5522 & 6.2531e-04 & 0.6609 & 2.3338e-03 & 1.4005 \\
		& 160 & 6.5764e-04 & 0.3736 & 5.8745e-04 & 0.4678 & 1.9451e-03 & 1.3645 \\
		& 180 & 6.3746e-04 & 0.2646 & 5.6431e-04 & 0.3411 & 1.6649e-03 & 1.3204 \\
		& 200 & 6.2424e-04 & 0.1990 & 5.4906e-04 & 0.2600 & 1.4565e-03 & 1.2694 \\
		& 220 & 6.1483e-04 & 0.1594 & 5.3828e-04 & 0.2081 & 1.2975e-03 & 1.2128 \\
		$\ssq=0.01$
		& 60   & 1.4210e-03 & --      & 1.4497e-03 & --      & 8.0173e-03 & --      \\
		& 80   & 8.4287e-04 & 1.8156 & 8.5484e-04 & 1.8360 & 5.2694e-03 & 1.4589 \\
		& 100  & 5.6521e-04 & 1.7909 & 5.7091e-04 & 1.8091 & 3.7981e-03 & 1.4672 \\
		& 120  & 4.1217e-04 & 1.7319 & 4.1441e-04 & 1.7572 & 2.9041e-03 & 1.4720 \\
		& 140 & 3.2039e-04 & 1.6341 & 3.2005e-04 & 1.6761 & 2.3136e-03 & 1.4746 \\
		& 160 & 2.6227e-04 & 1.4991 & 2.5968e-04 & 1.5652 & 1.8999e-03 & 1.4755 \\
		& 180 & 2.2409e-04 & 1.3359 & 2.1946e-04 & 1.4287 & 1.5968e-03 & 1.4752 \\
		& 200 & 1.9834e-04 & 1.1586 & 1.9187e-04 & 1.2754 & 1.3672e-03 & 1.4736 \\
		& 220 & 1.8061e-04 & 0.9822 & 1.7251e-04 & 1.1159 & 1.1884e-03 & 1.4710 \\
		$\ssq=0.0025$
		& 60   & 1.0500e-03 & --      & 1.2158e-03 & --      & 8.1234e-03 & --      \\
		& 80   & 6.1490e-04 & 1.8601 & 7.0359e-04 & 1.9012 & 5.3457e-03 & 1.4546 \\
		& 100  & 4.0609e-04 & 1.8593 & 4.6045e-04 & 1.9001 & 3.8557e-03 & 1.4642 \\
		& 120  & 2.8958e-04 & 1.8546 & 3.2591e-04 & 1.8956 & 2.9490e-03 & 1.4704 \\
		& 140 & 2.1778e-04 & 1.8484 & 2.4355e-04 & 1.8897 & 2.3493e-03 & 1.4748 \\
		& 160 & 1.7030e-04 & 1.8416 & 1.8939e-04 & 1.8833 & 1.9285e-03 & 1.4781 \\
		& 180 & 1.3720e-04 & 1.8449 & 1.5183e-04 & 1.8769 & 1.6199e-03 & 1.4806 \\
		& 200 & 1.1316e-04 & 1.8284 & 1.2467e-04 & 1.8706 & 1.3856e-03 & 1.4826 \\
		& 220 & 9.5119e-05 & 1.8224 & 1.0437e-04 & 1.8646 & 1.2029e-03 & 1.4842 \\
		\bottomrule
	\end{tabular}
\end{table}

\begin{remark}
Recall that our error analysis (in particular, the derivation of~\eqref{T3}) needed the lower bound
$\lambda_2\ge C_1\varepsilon$ on the penalty parameter~$\lambda_2$.
Numerical tests show that whether one imposes this condition or one takes $\lambda_2=0$
makes little difference to the numerical errors and rates of convergence in
the three errors of~\eqref{three:errors}.
Thus it seems likely that the condition $\lambda_2\ge C_1\varepsilon$ is merely an artifact of our analysis,
but its implementation is easy and does not harm the accuracy of the numerical results. 
\end{remark}

\section{Concluding remarks}
\label{sec:conclusion}

In this paper we examined the LDG method on a Shishkin mesh, 
using tensor-product piecewise polynomials
of degree $k>0$, for
a singularly perturbed convection-diffusion problem on the unit square in~$\mathbb{R}^2$
whose solution exhibits characteristic and exponential boundary layers.
We obtained a $O((N^{-1}\ln N)^{k+1/2})$ energy-norm bound,
uniformly in the singular perturbation parameter,
for the error between the computed solution and the true solution.
Furthermore, we derived a $O((N^{-1}\ln N)^{k+1})$ supercloseness bound
for the energy-norm difference between the computed solution 
and a local Gauss-Radau projection of the true solution into the finite element space.
As a consequence of this supercloseness property, 
one obtains an optimal $O(N^{-(k+1)})$ convergence rate  for the $L^2$ error
of the computed solution.
These results are based on a large number of technical estimates 
of the approximations from the finite element space of various solution components 
on different subregions of the domain. 
Numerical experiments show that our theoretical results are sharp.

\section*{Statements and Declarations}
No competing interests.

%


\bibliography{ChengStynes}
\bibliographystyle{amsplain}

\end{document}